\documentclass[print,wide]{draft}

\graphicspath{{images/}}
\usepackage{booktabs,multirow}
\usepackage{placeins}

\begin{document}
\title{Stability Estimates for the Inverse Recovery of Drift and Diffusion Point Sources in Stochastic Parabolic Equations}

\vspace{-2mm}

\author{Qi L\"u\thanks{School of Mathematics, Sichuan University, Chengdu 610064, China. {\small\it E-mail:} {\small\tt lu@scu.edu.cn.} This author is supported by the National Natural Science Foundation of China under grant 12025105.}, \quad
Yu Wang\thanks{School of Mathematics, Southwest Jiaotong University, Chengdu 611756, China. {\small\it E-mail:} {\small\tt yuwangmath@163.com.} This author is supported by the National Natural Science Foundation of China under grant
12401589 and by Sichuan Science and Technology Program under grant
2026NSFSC0777.}}

\date{}
\maketitle

\vspace{-6mm}

\begin{abstract}
This paper studies the simultaneous recovery of point-source locations and temporal strengths in both the drift and diffusion terms of a stochastic parabolic equation with Neumann boundary conditions, using only boundary observations on a nonempty portion of the boundary after the sources have become inactive. By separating the stochastic equation into its expectation and centered parts, the drift channel reduces to a deterministic source problem with scalar Laplace moments, while the diffusion channel is controlled by It\^o's isometry, which preserves the full \(L^2\)-energy of the integrand. Under suitable nondegeneracy and separation assumptions, we establish Lipschitz stability for the location of a single source in either channel, logarithmic stability for the drift strength, and Lipschitz stability for the complete diffusion strength including its sign. For multiple sources, we obtain Lipschitz stability up to permutation for drift locations and weighted temporal moments in dimensions two and three, and for diffusion locations and strengths in dimensions one through three. The proof employs explicit adjoint probes, including complex-isotropic polynomial annihilators in dimensions two and three and spectral synthesis in one dimension, combined with boundary-to-source estimates derived from Green's identities. These results extend and improve upon existing deterministic stability theories by revealing the fundamentally different temporal information carried by stochastic integrals, a distinction that is further illustrated by numerical experiments.
\end{abstract}


\vspace{1mm}

{\bf 2020 MR Subject Classification}: 35R30, 35R60 

\vspace{-3mm}

\section{Introduction}
\label{sec:introduction}

Point-source models are used when a localized release, such as in heat generation or pollutant transport, is small compared with the observation length scale. Such models arise, for instance, in pollutant-source identification, where time-dependent point releases are described by advection-dispersion-reaction equations \cite{ElBadiaHaDuongHamdi2005}, and white-noise-modulated sources have been used in diffusion-convection models of turbulent transport \cite{ElTawil1998}. In this context, the drift and diffusion sources represent mean localized releases and zero-mean fluctuations, respectively, about a background concentration. The corresponding inverse problem involves both finite-dimensional geometric unknowns, namely the source locations, and infinite-dimensional temporal unknowns, namely the release histories. In the deterministic setting, this mixed nature already leads to different stability regimes for these two types of unknowns. The goal of this paper is to investigate how this picture changes when both the drift and the Itô diffusion channels contain point sources in a stochastic parabolic equation.

Let $n\in\{1,2,3\}$,  $G\subset\mathbb R^n$ be a bounded connected
domain with boundary of class $C^2$,  $\Gamma\subset\partial G$ be nonempty and
relatively open, and $0<T_0<T_1<T$. 
Set $\Sigma=(0,T)\times\partial G$.
On a complete filtered probability space 
$(\Omega,\mathcal F,\mathbf F,\mathbb P)$ carrying a one-dimensional standard
Brownian motion $\{W(t)\}_{t\geq0}$ with augmented filtration 
$\mathbf F=\{\mathcal F_t\}_{t\geq0}$, let $\mathbb F$ denote the progressive $\sigma$-field associated with
$\mathbf F$. For a separable Hilbert space $\mathcal H$,
$L^2_{\mathbb F}(0,T;\mathcal H)$ denotes the space of progressively
measurable $\mathcal H$-valued processes $X$ such that
$\mathbb E\int_0^T\|X(t)\|_{\mathcal H}^2\,dt<\infty$, and
$L^2_{\mathbb F}(\Omega;C([0,T];\mathcal H))$ denotes the space of adapted
continuous processes satisfying
$\mathbb E\sup_{0\leq t\leq T}\|X(t)\|_{\mathcal H}^2<\infty$.
We use $L^2_{\mathbb F}(\Omega;H^1(T_1,T;\mathcal H))$ for the adapted
processes in $L^2(\Omega;H^1(T_1,T;\mathcal H))$.
Here and throughout, $\nu$ denotes the outward unit normal vector on
$\partial G$.

Fix known integers $N_0,N_1\geq0$ with $N_0+N_1\geq1$, and interpret
an empty sum as zero. We consider\vspace{-2mm}
\begin{equation}
\label{eq:model}
\begin{cases}
\begin{aligned}
 &d u+\mathcal{L} u\,dt
=\sum_{k=1}^{N_0}\lambda_k(t)\delta_{x_k}\,dt
+\sum_{l=1}^{N_1}\gamma_l(t)\delta_{y_l}\,dW(t)
&&\text{in }(0,T)\times G,\\
&\partial_\nu u=0
&&\text{on }(0,T)\times\partial G,\\
&u(0)=u_0
&&\text{in }G,
\end{aligned}
\end{cases}\vspace{-2mm}
\end{equation}
where\vspace{-2mm}
\begin{equation*}
\label{eq:operator}
\mathcal{L}=-\Delta+A\cdot\nabla+\mu,
\qquad A\in\mathbb{R}^n,\quad \mu\in\mathbb{R}.
\end{equation*}
Here $\delta_\xi$ denotes the Dirac measure at $\xi\in G$.
The initial datum $u_0\in L^2(G)$ is deterministic and known. The source
locations $x_k,y_l\in G$ and the deterministic strengths
$\lambda_k,\gamma_l\in L^2(0,T)$ are unknown. 

\vspace{-2mm}

\begin{problem}[Inverse point-source problem]
Determine the source locations and temporal strengths
$\{(x_k,\lambda_k)\}_{k=1}^{N_0}$ and
$\{(y_l,\gamma_l)\}_{l=1}^{N_1}$ from the boundary trace $u|_\Sigma$.
\end{problem}
\vspace{-2mm}

Identification and reconstruction problems for point sources in deterministic PDEs have been extensively studied. Early contributions established uniqueness for stationary sources that vanish prior to the end of the experiment, based on partial lateral boundary data \cite{ElBadiaHaDuong2002}. Subsequent work extended these results to time-dependent intensities in one dimension \cite{ElBadiaHaDuongHamdi2005}, moving sources \cite{AndrleBelgacemElBadia2011}, and multiple moving pollutants \cite{AndrleElBadia2012}. Ling and Takeuchi \cite{LingTakeuchi2009} recovered source counts, locations, and strengths from sparse fixed-point measurements. More recent developments include reconstruction methods for time-dependent emissions \cite{QiuWangYuYu2025}, Dirac sources from sparse data \cite{GuZhangZhang2025}, and sources from single-boundary observations \cite{SunWang2025}; uniqueness from finite flux measurements has also been addressed \cite{GongJinKianLiu2026}. While these results clarify identifiability issues, they do not provide the stability estimates required for error quantification in our context.

For regular sources, Lipschitz stability via Carleman estimates is by now classical \cite{ImanuvilovYamamoto1998}, whereas Dirac singularities demand special treatment. Komornik and Yamamoto \cite{KomornikYamamoto2005} derived sharp conditional estimates for point masses with prescribed temporal profiles. The most directly relevant work for our drift channel is \cite{HuangJinKianTriki2026}, which establishes Lipschitz stability for the location and logarithmic stability for the strength of single sources, as well as planar H\"older estimates for multiple sources under moment/matching assumptions, using Carleman estimates, Laplace analysis, and adjoint elliptic solutions. An optimal-transport framework for positive point-source measures appears in \cite{QiuYu2025} under finite-dimensional temporal normalization. Thus, while the deterministic theory reveals distinct stability scales for locations and histories, it does not capture the additional information carried by stochastic integrals in the diffusion channel.

Stochastic inverse source problems are more recent; see \cite{LuZhang2024,LuWang2026Book}. L\"u \cite{Lu2012} proved uniqueness via Carleman estimates.   \cite{Yuan2021} considered two additive source types. More recently, \cite{WangYangZhang2026} recovered temporal-factor moduli from partial flux, and \cite{LiLuQianWang2026} proved uniqueness of drift/diffusion supports via shape optimization. Despite these advances, quantitative stability for the simultaneous recovery of interior point locations and full temporal strengths in both channels, including signed diffusion profiles and finite-source matching, has not yet been established.

A key ingredient in the analysis below is the assumption that all sources become inactive before the end of the observation period. This switch-off condition is standard in the point-source literature \cite{ElBadiaHaDuong2002,ElBadiaHaDuongHamdi2005,HuangJinKianTriki2026} and is essential for obtaining the terminal-state regularity needed in the stability proof.

\begin{condition}
\label{con1}
The source amplitudes vanish after $T_0$:\vspace{-2mm}
\begin{equation}
\label{eq:switch-off}
\lambda_k(t)=0\quad(k=1,\ldots,N_0),\qquad
\gamma_l(t)=0\quad(l=1,\ldots,N_1),
\quad\text{for a.e. }t\in(T_0,T).\vspace{-2mm}
\end{equation}
\end{condition}

To state the main stability results, we compare two candidate
configurations with the same source counts,\vspace{-2mm}
\begin{equation*}
\{(x_k^j,\lambda_k^j)\}_{k=1}^{N_0},\qquad
\{(y_l^j,\gamma_l^j)\}_{l=1}^{N_1},\qquad j=1,2.\vspace{-2mm}
\end{equation*}
Both configurations satisfy \eqref{eq:switch-off}. Let $u^j$ be the
corresponding solutions with the same initial datum. For $n=1$, boundary
integrals use counting measure; in particular, $H^1(\Gamma)=L^2(\Gamma)$.

We first consider $N_0,N_1\in\{0,1\}$, with at least one channel present.
For each source that is present, we suppress the source index by
writing $x^j=x_1^j$, $y^j=y_1^j$, $\lambda^j=\lambda_1^j$, and
$\gamma^j=\gamma_1^j$ for $j=1,2$.

Quantifying drift-strength stability requires minimal regularity and an energy lower bound. We assume the first reference strength has nonzero 
$L^2$-norm and that the difference of two candidate strengths lies in a suitable Sobolev space for frequency-domain analysis.
\begin{assumption}
\label{ass:lambda}
There exist $s\in(0,\tfrac12)$ and $M,r>0$ such that
$\lambda^1\in H^s(0,T)$, $\lambda^1-\lambda^2\in H_0^1(0,T)$, and\vspace{-2mm}
\begin{equation}
\label{eq:lambda-assumptions}
\|\lambda^1\|_{H^s(0,T)}
+\|\lambda^1-\lambda^2\|_{H^1(0,T)} \leq M,\qquad
\|\lambda^1\|_{L^2(0,T)} \geq r.\vspace{-2mm}
\end{equation}
\end{assumption}

For the diffusion channel, the It\^o isometry permits a direct $L^2$-based estimate of the strength difference, so the required assumptions are milder. It suffices to impose a uniform upper bound on both strengths and a nondegeneracy lower bound on the reference strength.\vspace{-2mm}
\begin{assumption}
\label{ass:gamma}
There exist $M>0$ and $r>0$ such that\vspace{-2mm}
\begin{equation}
\label{eq:gamma-bounds}
\|\gamma^1\|_{L^2(0,T)}\geq r,
\qquad
\|\gamma^1\|_{L^2(0,T)}
+\|\gamma^2\|_{L^2(0,T)}\leq M.\vspace{-2mm}
\end{equation}
\end{assumption}
Define the logarithmic modulus\vspace{-4mm}
\begin{equation}
\label{eq:log-modulus}
\Phi(\eta)=
\begin{cases}
[\log(3+\eta^{-1})]^{-2},&\eta>0,\\
0,&\eta=0.
\end{cases}\vspace{-2mm}
\end{equation}
For a boundary trace $f$ for which the following quantities are finite, set\vspace{-2mm}
\begin{equation*}
\| f \|_{\operatorname{obs}} = 
\|f\|_{L^2_{\mathbb F}(0,T;L^2(\partial G))}
+\|f\|_{L^2_{\mathbb F}(T_1,T;H^1(\Gamma))}
+\|f\|_{L^2_{\mathbb F}(\Omega;H^1(T_1,T;L^2(\Gamma)))}.\vspace{-2mm}
\end{equation*}
For a state $u$, we write
$\|u\|_{\operatorname{obs}}=\|u|_\Sigma\|_{\operatorname{obs}}$.
For deterministic traces, the expectation is omitted.

Throughout Sections~\ref{sec:introduction}--\ref{sec:multiple}, $C$ denotes
a positive constant that may change from line to line. Unless otherwise
stated, its dependence on $G,\Gamma,T,T_0,T_1,A$, and $\mu$ is suppressed.
Any additional dependence is specified in the relevant statement or
locally in the proof.

\begin{theorem}
\label{thm:main-single}
Let $n\in\{1,2,3\}$. Under \cref{con1}, the following statements hold.
\begin{enumerate}[(a)]
\item \emph{Drift source.}
Let $N_0=1$, $x^j\in G$, and suppose that Assumption~\ref{ass:lambda}
holds. There exists $C=C(s,M,r)>0$ such that\vspace{-3mm}
\begin{align}
\label{eq:main-drift-location}
\lvert x^1-x^2\rvert
& \leq C\|u^1-u^2\|_{\operatorname{obs}},
\\
\label{eq:main-drift-amplitude}
\|\lambda^1-\lambda^2\|_{L^2(0,T)}
& \leq C\Phi\bigl(\|u^1-u^2\|_{\operatorname{obs}}\bigr).\vspace{-2mm}
\end{align}

\item \emph{Diffusion source.}
Let $N_1=1$, $y^j\in G$, and suppose that Assumption~\ref{ass:gamma}
holds. There exists $C=C(M,r)>0$ such that\vspace{-2mm}
\begin{equation}
\label{eq:main-diffusion}
\lvert y^1-y^2\rvert
+\|\gamma^1-\gamma^2\|_{L^2(0,T)}
\leq C\|u^1-u^2\|_{\operatorname{obs}}.\vspace{-2mm}
\end{equation}
\end{enumerate}
\end{theorem}

\begin{remark}[Comparison with deterministic single-source stability]
\label{rem:deterministic-single}
In the absence of the diffusion channel, \cref{thm:main-single}(a) recovers the Lipschitz location and logarithmic strength estimates of \cite[Theorems~2.1 and~2.3]{HuangJinKianTriki2026} for $n=3$ and $n=2$, respectively; for $n=1$, it removes the restriction $\mu=-|A|^2/4$ and replaces the nondegeneracy condition on $\int_0^T\lambda^1(t)\,dt$ in Theorem~2.4 by the $L^2(0,T)$ lower bound in \cref{ass:lambda}.
\end{remark}

For multiple source configurations, the drift estimate determines one
weighted temporal integral for each source, whereas the diffusion estimate
recovers the strengths in $L^2(0,T)$.

Set
$p_\sharp=-\mu-\frac{|A|^2}{4}$. 
For multiple drift sources, localization via polynomial annihilation requires a minimal separation distance between sources and a lower bound on the weighted temporal moments to be recovered.\vspace{-1mm}
\begin{assumption}[Multiple drift sources]
\label{ass:multi-drift}
Suppose that, for some $\delta,r>0$,\vspace{-2mm}
\begin{align}
\label{eq:multi-drift-separation}
|x_k^j-x_l^j|&\geq\delta &&(k\neq l,\ j=1,2),\\
\label{eq:multi-drift-moment-lower}
\bigg|\int_0^T\mathrm e^{-p_\sharp t}\lambda_k^1(t)\,dt\bigg|
&\geq r &&(k=1,\ldots,N_0).
\end{align}\vspace{-4mm}
\end{assumption}
Multiple diffusion sources require the same separation condition, an \(L^2\)-nondegeneracy lower bound on each reference strength, and a uniform total-energy upper bound on all candidates. \vspace{-2mm}
\begin{assumption}[Multiple diffusion sources]
\label{ass:multi}
Suppose that, for some $\delta,r,M>0$,
\begin{align}
\label{eq:sep}
&|y_l^j-y_m^j| \geq\delta &&(l\neq m,\ j=1,2),\\
\label{eq:multi-nondeg}
&\|\gamma_l^1\|_{L^2(0,T)} \geq r &&(l=1,\ldots,N_1),\\
\label{eq:multi-upper}
&\sum_{l=1}^{N_1}\bigl(\|\gamma_l^1\|_{L^2(0,T)}
+\|\gamma_l^2\|_{L^2(0,T)}\bigr) \leq M.
\end{align}
\end{assumption}

\begin{theorem}[Finite-source stability]
\label{thm:main-multi}
Under Condition~\ref{con1}, the following statements hold.
\begin{enumerate}[(a)]
\item \emph{Drift sources.}
Let $n\in\{2,3\}$ and $N_0\geq2$, and suppose that
Assumption~\ref{ass:multi-drift} holds. There exist a constant
$C=C(N_0,\delta,r)>0$ and a permutation $\sigma$ of
$\{1,\ldots,N_0\}$ such that\vspace{-2mm}
\begin{equation}
\label{eq:multi-drift-location-result}
\max_{1\leq k\leq N_0}|x_k^1-x_{\sigma(k)}^2|
\leq C\|u^1-u^2\|_{\operatorname{obs}}.\vspace{-2mm}
\end{equation}
If, in addition,\vspace{-3mm}
\begin{equation}
\label{eq:multi-drift-upper}
\sum_{k=1}^{N_0}
\left(\|\lambda_k^1\|_{L^2(0,T)}
+\|\lambda_k^2\|_{L^2(0,T)}\right)\leq M\vspace{-2mm}
\end{equation}
for some $M>0$, then the same permutation satisfies, with
$C=C(N_0,\delta,r,M)$,\vspace{-2mm}
\begin{equation}
\label{eq:multi-drift-moment-result}
\max_{1\leq k\leq N_0}
\bigg|\int_0^T\mathrm e^{-p_\sharp t}
\bigl(\lambda_k^1(t)-\lambda_{\sigma(k)}^2(t)\bigr)\,dt\bigg|
\leq C\|u^1-u^2\|_{\operatorname{obs}}.\vspace{-2mm}
\end{equation}

\item \emph{Diffusion sources.}
Let $n\in\{1,2,3\}$ and $N_1\geq2$, and suppose that
Assumption~\ref{ass:multi} holds. There exist a constant
$C=C(N_1,\delta,r,M)>0$ and a
permutation $\sigma$ of $\{1,\ldots,N_1\}$ such that\vspace{-2mm}
\begin{equation}
\label{eq:multi-result}
\max_{1\leq l\leq N_1}
\left(
|y_l^1-y_{\sigma(l)}^2|
+\|\gamma_l^1-\gamma_{\sigma(l)}^2\|_{L^2(0,T)}
\right)
\leq C\|u^1-u^2\|_{\operatorname{obs}}.\vspace{-2mm}
\end{equation}
\end{enumerate}
\end{theorem}

\begin{remark}[Comparison with deterministic finite-source stability]
\label{rem:deterministic-multiple}
In the planar case $\mu=-|A|^2/4$, we have $p_\sharp=0$, so
\eqref{eq:multi-drift-moment-lower} reduces to the moment condition in
\cite[Theorem~2.2]{HuangJinKianTriki2026}.
Under this specialization, \cref{thm:main-multi}(a) improves the
H\"older location estimate with exponent $1/N_0$ to a Lipschitz estimate
and removes the a priori nearest-neighbor matching condition imposed
there; see also \cref{rem:deterministic-matching}.

Moreover, it removes the coefficient restriction $\mu=-|A|^2/4$
and extends the result to dimension three, with nondegeneracy imposed
on the temporal moments weighted by $\mathrm e^{-p_\sharp t}$,
where $p_\sharp=-\mu-|A|^2/4$.
Under the additional bound \eqref{eq:multi-drift-upper}, the same
permutation gives Lipschitz recovery of the weighted temporal moments.
\end{remark}

\begin{remark}
The restriction $n\in\{2,3\}$ in \cref{thm:main-multi}(a) is essential:
in dimension one, two distinct configurations of two drift sources can
produce identical boundary traces, even under the switch-off condition
\eqref{eq:switch-off} and the moment condition
\eqref{eq:multi-drift-moment-lower}.
Such a counterexample follows by modifying the construction in
\cite[Section~4]{AndrleBelgacemElBadia2011} so that the auxiliary heat
solution returns to zero at $T_0$.
\end{remark}

The preceding counterexample exposes a fundamental limitation in one dimension while underscoring the importance of separation and nondegeneracy in higher dimensions. Our analysis refines the probe construction and localization of \cite{HuangJinKianTriki2026}. For a single drift source, a nonvanishing Laplace mode in a bounded frequency window, combined with frequency-dependent hyperbolic-sine probes, yields uniform one-dimensional estimates without the earlier coefficient and moment restrictions. For multiple sources, complex-isotropic polynomial probes at \(p_\sharp\) extend the planar construction to three dimensions and accommodate general \(\mu\). Separation ensures at most one small distance factor in the product estimate, giving Lipschitz localization and a matching permutation; normalized probes then recover the coefficients.

After separating the stochastic equation into expectation and centered parts, a stochastic Green identity and Itô's isometry control source evaluation in \(L^2(0,T)\), while deterministic testing controls scalar Laplace moments. This full-time integrand control yields Lipschitz recovery of signed diffusion strengths and permits combining finitely many spectral parameters into annihilating probes in one dimension, where finite-source diffusion recovery follows from the same localization and matching argument.

The remainder is organized as follows. \Cref{sec:forward} states the direct-problem regularity, separates channels, and derives boundary-to-source estimates. \Cref{sec:single-proof,sec:multiple} contain the single-source proof and the finite-source matching principle, probe constructions, and stability proof, respectively. \Cref{sec:numerics} presents numerical experiments, and Appendix~\ref{app:direct} establishes the direct problem and its boundary regularity.

\vspace{-2mm}

\section{Channel separation and boundary-to-source estimates}
\label{sec:forward}

The drift and diffusion channels reduce to fundamentally different estimates: deterministic scalar moments for the drift, and stochastic $L^2(0,T)$-valued estimates for the diffusion. Taking expectations removes the stochastic integral and leaves a deterministic equation for the mean, while the centered part is controlled by It\^o's isometry, which preserves the full $L^2$-energy of the integrand. The following subsections derive the boundary-to-source estimates for both channels.

Let $\mathcal A=-\mathcal L$ be the Neumann realization on $L^2(G)$, with\vspace{-2mm}
\begin{equation}
\label{eq:A-def}
\mathcal A u=\Delta u-A\cdot\nabla u-\mu u,
\quad  D(\mathcal A)=\{u\in H^2(G):\partial_\nu u=0\}, \quad 
\mathcal X_{-2}=D(\mathcal A^*)',\vspace{-2mm}
\end{equation}
where the dual is taken with pivot $L^2(G)$. We write $S(t)$ and
$S_{-2}(t)$ for the corresponding semigroups; their construction is given in
Appendix~\ref{app:direct}. The formal adjoint expression is
$\mathcal L^*=-\Delta-A\cdot\nabla+\mu$.

\begin{proposition}[Well-posedness and boundary regularity]
\label{prop:direct-wellposed}
For every fixed finite source configuration in \eqref{eq:model}, the direct
problem admits a unique mild solution
$
u\in L^2_{\mathbb F}(\Omega;C([0,T];\mathcal X_{-2}))$. 
Moreover,\vspace{-3mm}
\begin{equation}
\label{eq:wellposed-est}
\mathbb E\sup_{0\leq t\leq T}\|u(t)\|_{\mathcal X_{-2}}^2
\leq C\Bigl(
\|u_0\|_{L^2(G)}^2
+\sum_{k=1}^{N_0}\|\lambda_k\|_{L^2(0,T)}^2
+\sum_{l=1}^{N_1}\|\gamma_l\|_{L^2(0,T)}^2
\Bigr),\vspace{-2mm}
\end{equation}
where $C$ depends only on $G,A,\mu,T,N_0$, and $N_1$.
If Condition~\ref{con1} holds, then $\|u\|_{\operatorname{obs}}<\infty$ and 
$u\in L^2(\Omega;H^1(T_1,T;H^2(G)))$. 
In particular,\vspace{-3mm}
\begin{equation}
\label{eq:boundary-trace-estimate}
\|u\|_{\operatorname{obs}}^2\!
+\!\mathbb E\|u\|_{H^1(T_1,T;H^2(G))}^2
\! +\!\mathbb E\|u(T)\|_{H^2(G)}^2 \!\leq\! C_{\rm conf}\Bigl(
\|u_0\|_{L^2(G)}^2\!
+\!\sum_{k=1}^{N_0}\|\lambda_k\|_{L^2(0,T)}^2\!
+\!\sum_{l=1}^{N_1}\|\gamma_l\|_{L^2(0,T)}^2\Bigr).\vspace{-2mm}
\end{equation}
Here $C_{\rm conf}$ may additionally depend on $T_0,T_1$ and the source
locations, but not on the source strengths.
\end{proposition} 
The proof of \cref{prop:direct-wellposed} is given in Appendix~\ref{app:direct}.

For the two candidate solutions, set
$v=u^1-u^2$, $\bar v=\mathbb E v$ and $z=v-\bar v$. 
By  the zero-mean property of the
It\^o integral, the mean $\bar v$ is the solution of\vspace{-2mm}
\begin{equation}
\label{eq:mean-equation}
\begin{cases}
\begin{aligned}
& \partial_t \bar v+\mathcal{L} \bar v
=\sum_{k=1}^{N_0}\lambda_k^1(t)\delta_{x_k^1}
-\sum_{k=1}^{N_0}\lambda_k^2(t)\delta_{x_k^2}
&&\text{in }(0,T)\times G,\\
& \partial_\nu \bar v=0
&&\text{on }(0,T)\times\partial G,\\
& \bar v(0)=0
&&\text{in }G,
\end{aligned}
\end{cases}\vspace{-2mm}
\end{equation}
whereas $z$ is the mild solution of\vspace{-2mm}
\begin{equation}
\label{eq:center-equation}
\begin{cases}
\begin{aligned}
&d z+\mathcal{L} z\,dt
=\Big( \sum_{l=1}^{N_1}\gamma_l^1(t)\delta_{y_l^1}
- \sum_{l=1}^{N_1}\gamma_l^2(t)\delta_{y_l^2}\Big)
dW(t)
&&\text{in }(0,T)\times G,\\
& \partial_\nu z=0
&&\text{on }(0,T)\times\partial G,\\
& z(0)=0
&&\text{in }G.
\end{aligned}
\end{cases}\vspace{-2mm}
\end{equation}
The boundary regularity in \cref{prop:direct-wellposed} allows expectation
to commute with the boundary traces. 
It follows that\vspace{-2mm}
\begin{equation}
\label{eq:channel-contraction}
\max\bigl\{\|\bar v\|_{\operatorname{obs}},
 \|z\|_{\operatorname{obs}}\bigr\}
\leq \|v\|_{\operatorname{obs}}.\vspace{-2mm}
\end{equation}

\vspace{-2mm}

\subsection{Boundary-to-source estimates}

In this subsection, we derive boundary-to-source estimates for the two decoupled equations. A terminal-state estimate controls the terminal terms in Green's identities, yielding scalar Laplace-moment bounds for the drift and $L^2(0,T)$ bounds for the diffusion. These estimates are used in  \Cref{sec:single-proof,sec:multiple}.

\begin{lemma}[Terminal-state estimate]
\label{lem:terminal-estimate}
Let $w\in H^1(T_1,T;H^2(G))$ solve\vspace{-2mm}
\begin{equation*}
\partial_t w+\mathcal{L} w=0\quad\text{in }(T_1,T)\times G,
\qquad \partial_\nu w=0\quad\text{on }(T_1,T)\times\partial G.\vspace{-2mm}
\end{equation*}
Then\vspace{-2mm}
\begin{equation}
\label{eq:terminal-carleman}
\|w(T)\|_{H^2(G)}
\leq C\bigl(
\|w\|_{L^2(T_1,T;H^1(\Gamma))}
+\|w\|_{H^1(T_1,T;L^2(\Gamma))}\bigr).\vspace{-2mm}
\end{equation}
\end{lemma}

\begin{proof}
The estimate follows from the partial-boundary Carleman estimate and
parabolic smoothing, as in the proof of
\cite[Theorem~3.1, (3.2)]{HuangJinKianTriki2026}; see also
\cite[Lemma~2.4 and (2.12)]{ImanuvilovYamamoto1998}.
\end{proof}\vspace{-2mm}

For $q\in L^2(0,T)$, define its finite-time Laplace transform by\vspace{-2mm}
\begin{equation}
\label{eq:Laplace-def}
\widetilde q(p)=\int_0^T\mathrm e^{-pt}q(t)\,dt,
\qquad p\in\mathbb C.\vspace{-2mm}
\end{equation}
For $v\in H^2(G;\mathbb C)$, set\vspace{-2mm}
\begin{equation}
\label{eq:Bstar}
\mathcal B^*v=\partial_\nu v+(A\cdot\nu)v
\quad\text{on }\partial G.\vspace{-2mm}
\end{equation}
The $L^2(G)$ pairing is extended complex bilinearly, and no boundary
condition is imposed on $v$.
For $r\in\mathbb R$, write $r_+=\max\{r,0\}$.

\begin{lemma}[Deterministic boundary-to-source estimate]
\label{lem:det-source-estimate}
Let $w$ be the solution of\vspace{-2mm}
\begin{equation*}
\begin{cases}
\begin{aligned}
&\partial_t w+\mathcal Lw=\sum_{k=1}^Jq_k(t)\delta_{a_k}
&&\text{in }(0,T)\times G,\\
&\partial_\nu w=0
&&\text{on }(0,T)\times\partial G,\\
&w(0)=0
&&\text{in }G,
\end{aligned}
\end{cases}\vspace{-2mm}
\end{equation*}
where
$a_k\in G$ and $q_k\in L^2(0,T)$ vanish on $(T_0,T)$.
If $p\in\mathbb C$, $v\in H^2(G;\mathbb C)$, and
$(\mathcal L^*+p)v=0$, then\vspace{-2mm}
\begin{equation}
\label{eq:det-green}
\sum_{k=1}^J\widetilde q_k(p)v(a_k)
=\mathrm e^{-pT}(w(T),v)_{L^2(G)} +\int_0^T\!\int_{\partial G}
w(t,\xi)\mathrm e^{-pt}\mathcal B^*v(\xi)\,d\sigma\,dt.\vspace{-2mm}
\end{equation}
Moreover,\vspace{-4mm}
\begin{equation}
\label{eq:det-master}
\bigg|\sum_{k=1}^J\widetilde q_k(p)v(a_k)\bigg|
\leq C\mathrm e^{(-\operatorname{Re}p)_+T}
\|v\|_{H^2(G)}\|w\|_{\operatorname{obs}}.\vspace{-2mm}
\end{equation}
The constant is independent of $p$, the source locations, and the strengths.
\end{lemma}
\begin{proof}
Set $\widetilde v(t,\xi)=\mathrm e^{-pt}v(\xi)$. Since
$(\mathcal L^*+p)v=0$, we have
$-\partial_t\widetilde v+\mathcal L^*\widetilde v=0$.
Applying \cite[Lemma~4.1]{HuangJinKianTriki2026} with this
test function and using $w(0)=0$ gives \eqref{eq:det-green}.

We next estimate the two terms on its right-hand side. Since the sources
vanish after $T_0$, \cref{lem:terminal-estimate} gives
$\|w(T)\|_{H^2(G)}\leq C\|w\|_{\operatorname{obs}}$. Hence\vspace{-2mm}
\begin{equation}\label{9.8-eq6}
\bigl|\mathrm e^{-pT}(w(T),v)_{L^2(G)}\bigr|
\leq \mathrm e^{(-\operatorname{Re}p)_+T}
\|w(T)\|_{L^2(G)}\|v\|_{L^2(G)} \leq C\mathrm e^{(-\operatorname{Re}p)_+T}
\|v\|_{H^2(G)}\|w\|_{\operatorname{obs}}.\vspace{-2mm}
\end{equation}
For the boundary term, the Cauchy--Schwarz inequality and the trace
bound $\|\mathcal B^*v\|_{L^2(\partial G)}\leq C\|v\|_{H^2(G)}$ yield\vspace{-2mm}
\begin{equation*} 
\left|\int_0^T\!\!\int_{\partial G}\!
w(t,\xi)\mathrm e^{-pt}\mathcal B^*v(\xi) d\sigma dt\right|
\leq \sqrt T\,\mathrm e^{(-\operatorname{Re}p)_+T}
\|w\|_{L^2(0,T;L^2(\partial G))}
\|\mathcal B^*v\|_{L^2(\partial G)} \leq C\mathrm e^{(-\operatorname{Re}p)_+T}
\|v\|_{H^2(G)}\|w\|_{\operatorname{obs}}.\vspace{-2mm}
\end{equation*}
This, together with \eqref{9.8-eq6} and  \eqref{eq:det-green}, proves
\eqref{eq:det-master}.
\end{proof}

For the diffusion channel, define the source evaluation \vspace{-2mm}
\begin{equation}
\label{eq:source-evaluation-functional}
\mathfrak D[v](t):=\sum_{l=1}^{N_1}\gamma_l^1(t)v(y_l^1)
-\sum_{l=1}^{N_1}\gamma_l^2(t)v(y_l^2).\vspace{-2mm}
\end{equation}
\begin{lemma}[Stochastic boundary-to-source estimate]
\label{lem:sto-source-estimate}
Let $z$ be the solution of \eqref{eq:center-equation} under
Condition~\ref{con1}.
If $p\in\mathbb R$, $v\in H^2(G;\mathbb C)$, and
$(\mathcal L^*+p)v=0$, then\vspace{-2mm}
\begin{equation}
\label{eq:finite-source-green}
\int_0^T\mathrm e^{-pt}\mathfrak D[v](t)\,dW(t)
=\mathrm e^{-pT}(z(T),v)_{L^2(G)} 
+\int_0^T\!\int_{\partial G}
z\mathrm e^{-pt}\mathcal B^*v\,d\sigma\,dt,\qquad{\mathbb P}\mbox{-a.s.}\vspace{-2mm}
\end{equation}
and\vspace{-2mm}
\begin{equation}
\label{eq:master}
\|\mathfrak D[v]\|_{L^2(0,T)}
\leq C_p\|v\|_{H^2(G)}\|z\|_{\operatorname{obs}}.\vspace{-1mm}
\end{equation}
Here $C_p=C(p)$ is independent of the source locations and strengths.
\end{lemma}
\vspace{-2mm}
\begin{proof}
We regularize the point sources using
\vspace{-2mm}
\begin{equation}
\label{eq:source-mollifier}
\rho(\xi)=
\begin{cases}
c_n\exp\bigl(-1/(1-|\xi|^2)\bigr),&|\xi|<1,\\
0,&|\xi|\geq1,
\end{cases}
\qquad
\rho_\epsilon^y(\xi)=\epsilon^{-n}
\rho\Bigl(\frac{\xi-y}{\epsilon}\Bigr),\vspace{-2mm}
\end{equation}
where $c_n>0$ is chosen so that $\int_{\mathbb R^n}\rho(\xi)\,d\xi=1$.
For sufficiently small $\epsilon>0$, all $\rho_\epsilon^{y_l^j}$ are
supported in $G$. Let $z_\epsilon$ solve\vspace{-2mm}
\begin{equation}
\label{eq:regularized-center-equation}
\begin{cases}
\begin{aligned}
&d z_\epsilon+\mathcal Lz_\epsilon\,dt
=\Bigl(\sum_{l=1}^{N_1}\gamma_l^1(t)\rho_\epsilon^{y_l^1}
-\sum_{l=1}^{N_1}\gamma_l^2(t)\rho_\epsilon^{y_l^2}\Bigr)\,dW(t)
&&\text{in }(0,T)\times G,\\
&\partial_\nu z_\epsilon=0
&&\text{on }(0,T)\times\partial G,\\
&z_\epsilon(0)=0
&&\text{in }G.
\end{aligned}
\end{cases}\vspace{-2mm}
\end{equation}
Set\vspace{-3mm}
\begin{equation*}
\mathfrak D_\epsilon[v](t)
:=\sum_{l=1}^{N_1}\gamma_l^1(t)(\rho_\epsilon^{y_l^1},v)_{L^2(G)}
-\sum_{l=1}^{N_1}\gamma_l^2(t)(\rho_\epsilon^{y_l^2},v)_{L^2(G)}.\vspace{-2mm}
\end{equation*}
Applying  It\^o's formula and Green's formula to this
equation yields\vspace{-2mm}
\begin{equation*}
d\bigl(\mathrm e^{-pt}(z_\epsilon(t),v)_{L^2(G)}\bigr)
=\mathrm e^{-pt}\mathfrak D_\epsilon[v](t)\,dW(t)
-\mathrm e^{-pt}\int_{\partial G}z_\epsilon(t,\xi)
\mathcal B^*v(\xi)\,d\sigma\,dt.\vspace{-2mm}
\end{equation*}
Here the volume term vanishes because $(\mathcal L^*+p)v=0$, and
the deterministic test function contributes no quadratic covariation.
Integrating in time and using $z_\epsilon(0)=0$ gives\vspace{-2mm}
\begin{equation*}
\int_0^T\mathrm e^{-pt}\mathfrak D_\epsilon[v](t)\,dW(t)
=\mathrm e^{-pT}(z_\epsilon(T),v)_{L^2(G)}
+\int_0^T\!\int_{\partial G}
z_\epsilon(t,\xi)\mathrm e^{-pt}\mathcal B^*v(\xi)\,d\sigma\,dt.\vspace{-2mm}
\end{equation*}
By \cref{lem:point-source-approximation} in Appendix~\ref{app:direct},
all terms converge in $L^2(\Omega;\mathbb C)$ as $\epsilon\to0$.
Passing to the limit yields
\eqref{eq:finite-source-green}.

We now estimate the source evaluation. Since the sources vanish after
$T_0$, $z(\omega)$ satisfies the homogeneous parabolic equation on
$(T_1,T)\times G$ for almost every $\omega$.
Applying \cref{lem:terminal-estimate} to each sample path, with a constant
independent of $\omega$, and taking $L^2(\Omega)$ norms gives\vspace{-2mm}
\begin{equation*}
\|z(T)\|_{L^2(\Omega;L^2(G))}
\leq C\bigl(
\|z\|_{L^2_{\mathbb F}(T_1,T;H^1(\Gamma))}
+\|z\|_{L^2_{\mathbb F}(\Omega;H^1(T_1,T;L^2(\Gamma)))}\bigr) 
\leq C\|z\|_{\operatorname{obs}}.\vspace{-2mm}
\end{equation*}
Using It\^o's isometry, the Cauchy--Schwarz inequality, and the trace
bound $\|\mathcal B^*v\|_{L^2(\partial G)}\leq C\|v\|_{H^2(G)}$, we obtain\vspace{-4mm}
$$
\begin{aligned}
\|\mathrm e^{-p\cdot}\mathfrak D[v]\|_{L^2(0,T)}
&=\Big(\mathbb E\Big|\int_0^T\mathrm e^{-pt}
\mathfrak D[v](t)\,dW(t)\Big|^2\Big)^{1/2}\\
&\leq \mathrm e^{(-p)_+T}\Bigl(
\|z(T)\|_{L^2(\Omega;L^2(G))}\|v\|_{L^2(G)}
+\Bigl\|\int_0^T\|z(t)\|_{L^2(\partial G)}\,dt\Bigr\|_{L^2(\Omega)}
\|\mathcal B^*v\|_{L^2(\partial G)}\Bigr)\\
&\leq \mathrm e^{(-p)_+T}\Bigl(
\|z(T)\|_{L^2(\Omega;L^2(G))}\|v\|_{L^2(G)} 
+\sqrt T\,\|z\|_{L^2_{\mathbb F}(0,T;L^2(\partial G))}
\|\mathcal B^*v\|_{L^2(\partial G)}\Bigr)\\
&\leq C_p\|v\|_{H^2(G)}\|z\|_{\operatorname{obs}}.
\end{aligned}\vspace{-2mm}
$$
Finally, $\mathrm e^{-pt}\geq\mathrm e^{-p_+T}$ on $[0,T]$, so removing
the time weight gives \eqref{eq:master}.
\end{proof}

\vspace{-4mm}

\section{Proof of the single-source theorem}
\label{sec:single-proof}

Throughout this section, we take one source in the channel under
consideration, assume Condition~\ref{con1}, and use the notation of
\cref{thm:main-single}.
The same explicit spatial probes are used in both channels. The proofs differ
only in the temporal coercivity supplied by the two boundary-to-source
estimates.

\vspace{-2mm}

\subsection{Explicit adjoint probes}

We construct a location probe that vanishes at one candidate source point
and an amplitude probe normalized to one at that point. The location
probe has different forms in dimensions $n=2,3$ and $n=1$.

For $p\in\mathbb C$ and $v(\xi)=\mathrm e^{-A\cdot\xi/2}h(\xi)$,
direct differentiation gives\vspace{-2mm}
\begin{equation}
\label{eq:gauge-computation}
(\mathcal L^*+p)v
=\mathrm e^{-A\cdot\xi/2}
\biggl[-\Delta h+\biggl(p+\mu+\frac{|A|^2}{4}\biggr)h\biggr].\vspace{-2mm}
\end{equation}
Choose a  $\kappa$ satisfying\vspace{-4mm}
\begin{equation}
\label{eq:kappa}
\kappa^2=p+\mu+\frac{\lvert A\rvert^2}{4}.\vspace{-2mm}
\end{equation}
\paragraph{Dimensions $n=2,3$.}
For each unit vector $\varrho\in\mathbb R^n$, choose a unit
vector $f_{\varrho}\perp\varrho$. For $a,b\in G$, define\vspace{-2mm}
\begin{equation}
\label{eq:location-test}
v_{\varrho,a,b}^{p}(\xi)
=\mathrm{e}^{-\frac12 A\cdot(\xi-a)}
(\varrho\cdot(\xi-b))
\mathrm{e}^{\kappa f_{\varrho}\cdot(\xi-a)}.\vspace{-2mm}
\end{equation}
By \eqref{eq:gauge-computation} and $f_{\varrho}\perp\varrho$,\vspace{-3mm}
\begin{equation*}
(\mathcal L^*+p)v_{\varrho,a,b}^p=0,
\qquad v_{\varrho,a,b}^p(b)=0,
\qquad v_{\varrho,a,b}^p(a)=\varrho\cdot(a-b).\vspace{-2mm}
\end{equation*}
For $a\ne b$, choosing $\varrho=(a-b)/|a-b|$ gives
$v_{\varrho,a,b}^p(a)=|a-b|$.

\paragraph{Dimension $n=1$.}
After a translation, write $G=(0,\ell)$. For $a,b\in G$, define\vspace{-2mm}
\begin{equation}
\label{eq:location-test-1d}
\vartheta_{a,b}^{p}(\xi)
=\mathrm{e}^{-\frac A2(\xi-a)}
\begin{cases}
\dfrac{\sinh(\kappa(\xi-b))}{\kappa},&\kappa\ne0,\\[2mm]
\xi-b,&\kappa=0.
\end{cases}\vspace{-2mm}
\end{equation}
This probe satisfies\vspace{-2mm}
\begin{equation*}
(\mathcal L^*+p)\vartheta_{a,b}^p=0,
\qquad \vartheta_{a,b}^p(b)=0.\vspace{-2mm}
\end{equation*}
For a uniform lower bound at $a$, let $\mathfrak P$ be a compact subset of\vspace{-2mm}
\begin{equation}
\label{eq:p-compact-1d}
\Bigl\{p\in\mathbb C:
\operatorname{Re}\Bigl(p+\mu+\frac{A^2}{4}\Bigr)>0\Bigr\},\vspace{-2mm}
\end{equation}
and choose $\operatorname{Re}\kappa>0$. The function
$\sinh(\kappa r)/(\kappa r)$, continued by $1$ at $r=0$, is continuous
and nonzero on the compact set of admissible roots and
$r\in[-\operatorname{diam}G,\operatorname{diam}G]$. Hence\vspace{-2mm}
\begin{equation}
\label{eq:one-dimensional-probe-lower}
|\vartheta_{a,b}^p(a)|\geq c_{\mathfrak P}|a-b|,
\qquad a,b\in G,\quad p\in\mathfrak P,\vspace{-2mm}
\end{equation}
where $c_{\mathfrak P}>0$ depends only on $G,A,\mu,\mathfrak P$.
\vspace{-2mm}
\paragraph{A common amplitude probe.}
For every $n\in\{1,2,3\}$, $p\in\mathbb C$, $b\in G$, and unit vector
$\varpi\in\mathbb R^n$, define\vspace{-2mm}
\begin{equation}
\label{eq:amplitude-test}
\psi_{b,\varpi}^{p}(\xi)
=\mathrm{e}^{-\frac12A\cdot(\xi-b)}
\mathrm{e}^{\kappa\varpi\cdot(\xi-b)}.\vspace{-2mm}
\end{equation}
We take $\varpi=1$ when $n=1$. By \eqref{eq:gauge-computation}, 
$
(\mathcal L^*+p)\psi_{b,\varpi}^p=0$ and $\psi_{b,\varpi}^p(b)=1$. 
For each probe in its stated dimension, differentiating the explicit
formula at most twice gives \vspace{-2mm}
\begin{equation}
\label{eq:probe-growth}
\|\varphi\|_{H^2(G)} \leq C(1+|\kappa|)^2\mathrm e^{C|\kappa|},
\qquad \varphi\in\{v_{\varrho,a,b}^p,\,\vartheta_{a,b}^p,\,\psi_{b,\varpi}^p\}.\vspace{-2mm}
\end{equation}
The mean-value formula for the exponential also yields\vspace{-2mm}
\begin{equation}
\label{eq:probe-close}
|\psi_{b,\varpi}^p(a)-1|
 \leq C(1+|\kappa|)\mathrm e^{C|\kappa|}|a-b|.\vspace{-2mm}
\end{equation}
Here $C$ depends only on $G$ and $A$, uniformly in the unit vectors.

\vspace{-2mm} 

\subsection{Proof of the single-source stability estimates}

\begin{proof}[Proof of \cref{thm:main-single}]
Set $\varepsilon=\|v\|_{\operatorname{obs}}$, with $v,\bar v,z$ as in
\Cref{sec:forward}, and fix
$p_0=\max\{0,-\mu-|A|^2/4\}+2$.
By \eqref{eq:channel-contraction}, both channel errors are at most
$\varepsilon$.

\medskip\noindent
\textbf{Part I: drift channel.}
In this part, $C$ may additionally depend on $s$, $M$, and $r$.
Take $w=\bar v$ in \cref{lem:det-source-estimate}.

\smallskip\noindent
\emph{Step 1: selection of a nondegenerate frequency.}
The zero-extension map $H^s(0,T)\to H^s(\mathbb R)$ is bounded for
$0<s<1/2$ \cite[Lemma~3.32 and Theorem~3.33]{McLean2000}. 
Thus, by Assumption~\ref{ass:lambda},
$f_1(t)=\mathrm e^{-p_0t}\lambda^1(t)\mathbf1_{(0,T)}(t)$ satisfies
$\|f_1\|_{H^s(\mathbb R)}\leq CM$ and
$\|f_1\|_{L^2(\mathbb R)}\geq\mathrm e^{-p_0T}r$.

With the Fourier convention
$\widehat f_1(\tau)=\int_{\mathbb R}\mathrm e^{-\mathrm i\tau t}f_1(t)\,dt$,
we have $\widehat f_1(\tau)=\widetilde{\lambda^1}(p_0+\mathrm i\tau)$.
Plancherel's theorem and the $H^s$ tail bound give\vspace{-3mm}
\begin{equation*}
\int_{-R}^{R}|\widehat f_1(\tau)|^2\,d\tau
\geq 2\pi\mathrm e^{-2p_0T}r^2-CM^2R^{-2s},\qquad R\geq1.\vspace{-2mm}
\end{equation*}
Choosing $R_0\geq1$ so that
$CM^2R_0^{-2s}\leq\pi\mathrm e^{-2p_0T}r^2$ yields \vspace{-3mm}
\begin{equation*}
\sup_{|\tau|\leq R_0}|\widetilde{\lambda^1}(p_0+\mathrm i\tau)|
\geq\left(\frac{\pi\mathrm e^{-2p_0T}r^2}{2R_0}\right)^{1/2}=:r_0>0.\vspace{-2mm}
\end{equation*}
Here $R_0$ and $r_0$ depend only on $p_0,T,s,M,r$.

Choose $\tau_0\in[-R_0,R_0]$ such that\vspace{-2mm}
\begin{equation}
\label{eq:tau0-lower}
\bigl|\widetilde{\lambda^1}(p_0+\mathrm{i}\tau_0)\bigr|\geq r_0/2.\vspace{-2mm}
\end{equation}
Set $p_\dagger=p_0+\mathrm{i}\tau_0$.

\smallskip\noindent
\emph{Step 2: location recovery.}
The location estimate is immediate when $x^1=x^2$. Assume $x^1\ne x^2$.

\textbf{Case 1.} $n=2,3$. Choose $\varrho=(x^1-x^2)/|x^1-x^2|$
and use $v_{\varrho,x^1,x^2}^{p_\dagger}$. By
\cref{eq:det-master,eq:probe-growth},\vspace{-2mm}
\begin{equation}\label{9.8-eq1}
|\widetilde{\lambda^1}(p_\dagger)|\,|x^1-x^2|\leq C\varepsilon,\vspace{-2mm}
\end{equation}
with $C$ uniform for $\tau_0\in[-R_0,R_0]$.

\textbf{Case 2.} $n=1$. The compact frequency segment \(\mathfrak P_0=\{p_0+\mathrm{i}\tau:|\tau|\leq R_0\}\) lies in the set in \eqref{eq:p-compact-1d} by the choice of \(p_0\). Using \(\vartheta_{x^1,x^2}^{p_\dagger}\), which vanishes at \(x^2\), \cref{lem:det-source-estimate} together with \cref{eq:one-dimensional-probe-lower,eq:probe-growth} yields\vspace{-2mm}
\begin{equation}\label{9.8-eq2}
c_{\mathfrak P_0}|\widetilde{\lambda^1}(p_\dagger)|\,|x^1-x^2|
\leq C\varepsilon.\vspace{-2mm}
\end{equation}
In both \eqref{9.8-eq1} and \eqref{9.8-eq2}, \eqref{eq:tau0-lower} gives\vspace{-2mm}
\begin{equation}
\label{eq:det-stability-location}
|x^1-x^2|\leq C\varepsilon.\vspace{-2mm}
\end{equation}

\smallskip\noindent
\emph{Step 3: low-frequency strength estimate.}
Fix $R\geq1$ and $|\tau|\leq R$, and set $p=p_0+\mathrm{i}\tau$.
Use the common amplitude probe with $\varpi$ fixed as in its construction,
and set
$\beta(p)=\psi_{x^2,\varpi}^{p}(x^1)$. Since
$\psi_{x^2,\varpi}^{p}(x^2)=1$, applying \eqref{eq:det-master} yields\vspace{-2mm}
\begin{equation}
\label{eq:det-frequency-first}
\bigl|\beta(p)\widetilde{\lambda^1}(p)-\widetilde{\lambda^2}(p)\bigr|
\leq C\|\psi_{x^2,\varpi}^{p}\|_{H^2(G)}
\varepsilon.\vspace{-2mm}
\end{equation}
Furthermore,
$|\widetilde{\lambda^1}(p)|\leq\sqrt T\|\lambda^1\|_{L^2(0,T)}\leq C M$.
By \eqref{eq:probe-close} and \eqref{eq:det-stability-location},\vspace{-2mm}
\begin{equation*}
|\beta(p)-1|\,|\widetilde{\lambda^1}(p)|
\leq C(1+|\kappa|)\mathrm{e}^{C|\kappa|}
|x^1-x^2|M
\leq C(1+|\kappa|)\mathrm{e}^{C|\kappa|}\varepsilon.\vspace{-2mm}
\end{equation*}
Combining this estimate with \eqref{eq:det-frequency-first} and noting that
$|\kappa|\leq C(1+R)^{1/2}$ for $|\tau|\leq R$, we obtain\vspace{-2mm}
\begin{equation}
\label{eq:det-frequency}
\sup_{|\tau|\leq R}
\bigl|\widetilde{(\lambda^1-\lambda^2)}(p_0+\mathrm{i}\tau)\bigr|
\leq C(1+R)^2\mathrm{e}^{C\sqrt R}\varepsilon.\vspace{-2mm}
\end{equation}

\smallskip\noindent
\emph{Step 4: high-frequency tail and optimization.}
Extend $f(t)=\mathrm{e}^{-p_0t}\bigl(\lambda^1(t)-\lambda^2(t)\bigr)$
by zero outside $(0,T)$.
Then $\widehat f(\tau)=(\widetilde{\lambda^1-\lambda^2})(p_0+\mathrm i\tau)$.
Since $\lambda^1-\lambda^2\in H_0^1(0,T)$, boundedness of zero extension gives\vspace{-2mm}
\begin{equation}
\label{eq:f-H1}
\|f\|_{H^1(\mathbb R)}\leq C M.\vspace{-2mm}
\end{equation}
Plancherel's theorem, \eqref{eq:det-frequency}, and \eqref{eq:f-H1} imply that,
for $R\geq1$,\vspace{-2mm}
\begin{align}
\label{eq:det-optimization}
\|\lambda^1-\lambda^2\|_{L^2(0,T)}^2
&\leq C\int_{\mathbb R}|\widehat f(\tau)|^2\,d\tau \leq C\Bigl(
2R\sup_{|\tau|\leq R}|\widehat f(\tau)|^2
+\int_{|\tau|>R}|\widehat f(\tau)|^2\,d\tau\Bigr)\notag\\
&\leq C\bigl(
R(1+R)^4\mathrm{e}^{C\sqrt R}\varepsilon^2
+R^{-2}M^2\bigr)\\
&\leq C\bigl(
\mathrm{e}^{C_1\sqrt R}\varepsilon^2+R^{-2}
\bigr).\notag
\end{align} 
Here $C$ and $C_1$ are independent of $R$ and $\varepsilon$.

For $0<\varepsilon<\varepsilon_*$, choose
$\sqrt R=C_1^{-1}\log(\varepsilon^{-1})$, where the fixed threshold
$0<\varepsilon_*\leq1/2$ is small enough that $R\geq1$.
Then \eqref{eq:det-optimization} gives\vspace{-2mm}
\begin{equation*}
\|\lambda^1-\lambda^2\|_{L^2(0,T)}
\leq C[\log(\varepsilon^{-1})]^{-2}
\leq C\Phi(\varepsilon),\vspace{-1mm}
\end{equation*}
where we used \eqref{eq:log-modulus} and
$\log(3+\varepsilon^{-1})\leq3\log(\varepsilon^{-1})$
for $0<\varepsilon\leq1/2$.
If $\varepsilon=0$, letting $R\to\infty$ gives $\lambda^1=\lambda^2$,
consistent with $\Phi(0)=0$.
For $\varepsilon\geq\varepsilon_*$, we have
$\Phi(\varepsilon)\geq\Phi(\varepsilon_*)>0$, so Assumption~\ref{ass:lambda} gives \vspace{-2mm}
\begin{equation*}
\|\lambda^1-\lambda^2\|_{L^2(0,T)}
\leq M\leq\frac{M}{\Phi(\varepsilon_*)}\Phi(\varepsilon).\vspace{-2mm}
\end{equation*}
Increasing $C$ if necessary proves \eqref{eq:main-drift-amplitude}
for all $\varepsilon\geq0$.
Together with \eqref{eq:det-stability-location}, this proves part~(a).

\medskip\noindent
\textbf{Part II: diffusion channel.}
In this part, $C$ may additionally depend on $M$ and $r$.
At $p=p_0$, the square root $\kappa_0>0$ in \eqref{eq:kappa} is real.

\smallskip\noindent
\emph{Step 1: location recovery.}
The location estimate is immediate when $y^1=y^2$. Assume $y^1\ne y^2$.

\textbf{Case 1.} $n=2,3$. Choose
$\varrho=(y^1-y^2)/|y^1-y^2|$ and use $v_{\varrho,y^1,y^2}^{p_0}$.
It vanishes at $y^2$ and equals $|y^1-y^2|$ at $y^1$, so
\cref{lem:sto-source-estimate,eq:probe-growth} give\vspace{-2mm}
\begin{equation}\label{9.8-eq3}
|y^1-y^2|\|\gamma^1\|_{L^2(0,T)}\leq C\varepsilon.\vspace{-1mm}
\end{equation}

\textbf{Case 2.} $n=1$. Use the probe
$\vartheta_{y^1,y^2}^{p_0}$ from \eqref{eq:location-test-1d}.  It vanishes at
$y^2$. Applying \cref{lem:sto-source-estimate} with
\eqref{eq:probe-growth} and \eqref{eq:one-dimensional-probe-lower},
where $\mathfrak P=\{p_0\}$, gives\vspace{-2mm}
\begin{equation}\label{9.8-eq4}
c_{\{p_0\}}\lvert y^1-y^2\rvert
\|\gamma^1\|_{L^2(0,T)}
\leq C\varepsilon.\vspace{-1mm}
\end{equation}
In both \eqref{9.8-eq3} and \eqref{9.8-eq4}, the lower bound in \eqref{eq:gamma-bounds} yields\vspace{-2mm}
\begin{equation}
\label{eq:diff-location}
|y^1-y^2|\leq C\varepsilon.\vspace{-2mm}
\end{equation}

\smallskip\noindent
\emph{Step 2: strength recovery.}
Use $\psi_{y^2,\varpi}^{p_0}$ from \eqref{eq:amplitude-test}, with
$\varpi$ fixed as in its construction, and set
$\beta=\psi_{y^2,\varpi}^{p_0}(y^1)$. \Cref{lem:sto-source-estimate} gives\vspace{-2mm}
\begin{equation*}
\|\beta\gamma^1-\gamma^2\|_{L^2(0,T)}
\leq C\varepsilon.\vspace{-2mm}
\end{equation*}
By \eqref{eq:probe-close}, \eqref{eq:diff-location}, and the upper bound in
\eqref{eq:gamma-bounds},\vspace{-2mm}
\begin{equation*}
\|\gamma^1-\gamma^2\|_{L^2(0,T)}
\leq\|\beta\gamma^1-\gamma^2\|_{L^2(0,T)}
+\lvert \beta-1\rvert\|\gamma^1\|_{L^2(0,T)} \leq C\varepsilon+C\lvert y^1-y^2\rvert M
\leq C\varepsilon.\vspace{-2mm}
\end{equation*}
Combining this estimate with \eqref{eq:diff-location} proves
\eqref{eq:main-diffusion}.
\end{proof}

\vspace{-3mm}

\section{Finite-source stability}
\label{sec:multiple}

\vspace{-2mm}

This section establishes that probes isolating a source or a matched pair provide location‑coefficient stability, and then constructs such probes via polynomials in dimensions two and three, and via spectral synthesis for one‑dimensional diffusion sources. Here $N$ denotes the source count, with $N_0$ for drift and $N_1$ for diffusion.

\vspace{-2mm}

\subsection{From annihilating probes to stable matching}
\label{sec:matching-principle}

Consider two configurations $(a_k^j)_{k=1}^N$ in $G$, $j=1,2$.
To locate a target source, we cancel the contributions from all other
candidate points. After matching the locations, we instead retain the
two sources in each matched pair to compare their coefficients.
Write $c_k^j\in E$ for the coefficients, where $E$ is a complex normed
space. In the applications, $E=\mathbb C$ for drift moments and
$E=L^2(0,T;\mathbb C)$ for diffusion strengths. 

For a scalar function
$V:G\to\mathbb C$, set\vspace{-3mm}
\begin{equation*}
\mathfrak S[V]:=\sum_{k=1}^Nc_k^1V(a_k^1)
-\sum_{k=1}^Nc_k^2V(a_k^2).\vspace{-2mm}
\end{equation*}
Let $\varepsilon\geq0$ denote the observation error. The two types of
probe are specified as follows.

\smallskip\noindent
\emph{Location probes.}
For a target with $\min\limits_{1\leq m\leq N}|a_l^1-a_m^2|>0$,
we seek $H_l$ satisfying\vspace{-2mm}
\begin{equation}
\label{eq:abstract-product-probe}
\begin{cases}
\begin{aligned}
& H_l(a_k^1)=0\quad(k\ne l),\qquad
H_l(a_m^2)=0\quad(1\leq m\leq N),\\
&|H_l(a_l^1)|\geq c_0
\prod_{\substack{1\leq k\leq N\\k\ne l}}|a_l^1-a_k^1|
\prod_{m=1}^N|a_l^1-a_m^2|,\\
&\|\mathfrak S[H_l]\|_E\leq C_0\varepsilon.
\end{aligned}
\end{cases}\vspace{-2mm}
\end{equation}
Here $c_0,C_0>0$ are independent of $l$ and the configurations.
The zeros leave only $\mathfrak S[H_l]=c_l^1H_l(a_l^1)$, so the product
lower bound connects the observation error to the distances from the
target. The zero list contains the other $N-1$ points of the first
configuration and all $N$ points of the second, giving $2N-1$ entries.

\smallskip\noindent
\emph{Coefficient probes.}
After choosing a matching permutation $\sigma$, we seek probes $V_l$
that retain only the matched pair and equal $1$ at the reference point.
With $\|V\|_{\operatorname{Lip}(G)}=\|V\|_{L^\infty(G)}+
\sup_{\{a, b\in G,a\neq b\}}|V(a)-V(b)|/|a-b|$, we require\vspace{-2mm}
\begin{equation}
\label{eq:abstract-normalized-probe}
\begin{cases}
\begin{aligned}
&V_l(a_l^1)=1,\quad V_l(a_k^1)=0\ (k\ne l),\quad
V_l(a_m^2)=0\ (m\ne\sigma(l)),\\
&\|\mathfrak S[V_l]\|_E\leq C_1\varepsilon,\qquad
\|V_l\|_{\operatorname{Lip}(G)}\leq C_1.
\end{aligned}
\end{cases}\vspace{-1mm}
\end{equation}
Here $C_1>0$ is uniform in $l$ and the configurations.
The Lipschitz bound makes $V_l(a_{\sigma(l)}^2)$ close to $1$ when the
matched locations are close. This zero list omits both points of the
matched pair and has $2N-2$ entries. Either list may contain repeated
nodes when the configurations overlap.

The following lemma shows that these probe properties suffice for
stable recovery.

\begin{lemma}[Localization and matching]
\label{lem:localization-matching}
Assume that $|a_k^j-a_l^j|\geq\delta>0$ for $k\ne l$, $j=1,2$,
and $\|c_k^1\|_E\geq r>0$ for every $k$.
If every target with $\min\limits_{1\leq m\leq N}|a_l^1-a_m^2|>0$ admits a probe satisfying
\eqref{eq:abstract-product-probe}, then there exist
$C=C(G,N,\delta,r,c_0,C_0)>0$ and a permutation
$\sigma$ such that\vspace{-2mm}
\begin{equation}
\label{eq:abstract-location-result}
\max_{1\leq l\leq N}|a_l^1-a_{\sigma(l)}^2|\leq C\varepsilon.\vspace{-3mm}
\end{equation}
If, in addition, $\sum\limits_{k=1}^N(\|c_k^1\|_E+\|c_k^2\|_E)\leq M$
and every permutation with
$\max_l|a_l^1-a_{\sigma(l)}^2|<\delta/4$ admits probes satisfying
\eqref{eq:abstract-normalized-probe}, then the same permutation satisfies\vspace{-2mm}
\begin{equation}
\label{eq:abstract-coefficient-result}
\max_{1\leq l\leq N}\|c_l^1-c_{\sigma(l)}^2\|_E\leq C\varepsilon,\vspace{-2mm}
\end{equation}
where $C$ may additionally depend on $M$ and $C_1$.
\end{lemma}

\begin{proof}
For each $l\in\{1,\cdots,N\}$, choose a nearest point $a_{m_l}^2$. Separation of the second configuration gives\vspace{-2mm}
\begin{equation}
\label{eq:matching-nearest-separation}
|a_l^1-a_m^2|\geq\delta/2\qquad(m\ne m_l).\vspace{-2mm}
\end{equation}
Indeed, this holds immediately if $|a_l^1-a_{m_l}^2|\geq\delta/2$; otherwise,
$|a_l^1-a_m^2|\geq\delta-|a_l^1-a_{m_l}^2|>\delta/2$.
For $|a_l^1-a_{m_l}^2|>0$, the product-probe estimate and the reference lower bound yield\vspace{-2mm}
\begin{equation*}
c_0r\delta^{N-1}(\delta/2)^{N-1}|a_l^1-a_{m_l}^2|
\leq\|c_l^1\|_E|H_l(a_l^1)|
=\|\mathfrak S[H_l]\|_E\leq C_0\varepsilon.\vspace{-2mm}
\end{equation*}
Thus $|a_l^1-a_{m_l}^2|\leq C\varepsilon$ for every $l$, including those with $|a_l^1-a_{m_l}^2|=0$.
Choose $\varepsilon_*>0$ so that $C\varepsilon_*<\delta/4$.
For $\varepsilon\leq\varepsilon_*$, the nearest points are unique and
form a permutation $\sigma$: two points in the first configuration
assigned to the same point would be less than $\delta/2$ apart.
This proves \eqref{eq:abstract-location-result} for small errors.
For $\varepsilon>\varepsilon_*$, any permutation gives the location
estimate by boundedness of $G$.

Under the additional assumptions for coefficient recovery, consider
first $\varepsilon\leq\varepsilon_*$. For the matching above, every
nonmatched node is at distance at least $3\delta/4$ from $a_l^1$, and
the normalized probes give\vspace{-2mm}
\begin{equation*}
\mathfrak S[V_l]=c_l^1-c_{\sigma(l)}^2V_l(a_{\sigma(l)}^2).\vspace{-2mm}
\end{equation*}
Their Lipschitz bounds and $V_l(a_l^1)=1$ imply\vspace{-2mm}
\begin{equation*}
\|c_l^1-c_{\sigma(l)}^2\|_E
\leq\|\mathfrak S[V_l]\|_E
+\|c_{\sigma(l)}^2\|_E
|V_l(a_{\sigma(l)}^2)-V_l(a_l^1)| \leq C\varepsilon+CM|a_l^1-a_{\sigma(l)}^2|
\leq C\varepsilon.\vspace{-2mm}
\end{equation*}
For $\varepsilon>\varepsilon_*$, the bound $M$ gives the coefficient
estimate for the same permutation after increasing the constant.
\end{proof}

\vspace{-2mm} 

\subsection{Annihilating probes in dimensions two and three}
\label{sec:annihilating-probes}

We now construct the probes required by \cref{lem:localization-matching}
in dimensions two and three at the common spectral parameter $p_\sharp$.
The $H^2(G)$ bounds below allow us to apply the boundary-to-source
estimates in \cref{lem:det-source-estimate,lem:sto-source-estimate}.
The dot product is extended complex bilinearly, whereas $|\cdot|$
denotes the Euclidean norm on $\mathbb C^n$.

\begin{lemma}[Polynomial annihilators in dimensions two and three]
\label{lem:isotropic-annihilator}
Let $n\in\{2,3\}$, let $a\in G$, and let
$b_1,\ldots,b_J\in G\setminus\{a\}$, where $J\leq2N-1$.  There exists
$H_{a,\mathbf b}\in H^2(G;\mathbb C)$ such that\vspace{-2mm}
\begin{align}
\label{eq:isotropic-annihilator-equation}
&(\mathcal L^*+p_\sharp)H_{a,\mathbf b} =0,\\
&H_{a,\mathbf b}(b_j) =0,
\qquad j=1,\ldots,J,\label{eq:isotropic-annihilator-zeros}\\
&|H_{a,\mathbf b}(a)|
\geq c\prod_{j=1}^J|a-b_j|,
\qquad
\|H_{a,\mathbf b}\|_{H^2(G)}\leq C.
\label{eq:isotropic-annihilator-bounds}
\end{align}
Here $c,C>0$ depend only on $G,A,N$.  If, in addition,
$|a-b_j|\geq\rho>0$ for every $j$, then the normalized probe
$V_{a,\mathbf b}=H_{a,\mathbf b}/H_{a,\mathbf b}(a)$ satisfies\vspace{-2mm}
\begin{equation}
\label{eq:normalized-isotropic-bounds}
V_{a,\mathbf b}(a)=1,
\qquad V_{a,\mathbf b}(b_j)=0,
\qquad
\|V_{a,\mathbf b}\|_{H^2(G)}
+\|V_{a,\mathbf b}\|_{\operatorname{Lip}(G)}
\leq C_\rho,\vspace{-2mm}
\end{equation}
where $C_\rho$ depends only on $G$, $A$, $N$  and $\rho$.
\end{lemma}
\vspace{-2mm}
\begin{remark}
The planar construction below is the isotropic-vector formulation of
the holomorphic polynomial probes used in
\cite{HuangJinKianTriki2026}.  The point of \cref{lem:isotropic-annihilator} is its
quantitative extension to dimension three, obtained by choosing the
complex-isotropic direction according to the finite source
configuration.
\end{remark}

\begin{proof}[Proof of \cref{lem:isotropic-annihilator}]
\emph{Step 1: choosing a uniformly transverse isotropic direction.}
Set $\alpha_N=1-(2N)^{-1}$ and $c_N=(1-\alpha_N^2)^{1/2}\in(0,1)$.
We choose $\theta\in\mathbb C^n$ satisfying\vspace{-2mm}
\begin{equation*}
\theta\cdot\theta=0,\qquad |\theta|=\sqrt2,
\qquad |\theta\cdot(a-b_j)|\geq c_N|a-b_j|,
\quad j=1,\ldots,J.\vspace{-2mm}
\end{equation*}
The first condition makes polynomials in $\theta\cdot\xi$ harmonic.
The last gives a uniform lower bound at the target point.
For $n=2$, take $\theta=\varrho_1+\mathrm i\varrho_2$, where
$\varrho_1,\varrho_2$ are the standard coordinate vectors. Then
$|\theta\cdot(a-b_j)|=|a-b_j|$.

For $n=3$, set $u_j=(a-b_j)/|a-b_j|$ and let $\varsigma$ denote
normalized surface measure on $\mathbb S^2$. The spherical cap area formula gives\vspace{-2mm}
\begin{equation*}
\varsigma\{g\in\mathbb S^2:|g\cdot u_j|>\alpha_N\}
=1-\alpha_N=\frac1{2N}.\vspace{-2mm}
\end{equation*}
Since $J\leq2N-1$, the union of these exceptional sets has measure at
most $J/(2N)<1$. Thus there is $g\in\mathbb S^2$ with
$|g\cdot u_j|\leq\alpha_N$ for every $j$.
Take an orthonormal basis $\varrho,f$ of $g^\perp$ and set
$\theta=\varrho+\mathrm i f$. Then $\theta\cdot\theta=0$,
$|\theta|=\sqrt2$, and\vspace{-2mm}
\begin{equation*}
|\theta\cdot u_j|^2
=|\varrho\cdot u_j|^2+|f\cdot u_j|^2
=1-|g\cdot u_j|^2\geq1-\alpha_N^2=c_N^2.\vspace{-1mm}
\end{equation*}
For $J=0$, choose any direction satisfying the first two conditions;
empty products are interpreted as $1$.

\smallskip\noindent
\emph{Step 2: constructing the annihilator.}
Define\vspace{-2mm}
\begin{equation*}
P_{a,\mathbf b}(s)=\prod_{j=1}^J(s-\theta\cdot b_j),\qquad s\in\mathbb C,\qquad
 H_{a,\mathbf b}(\xi)
=\mathrm e^{-\frac12A\cdot(\xi-a)}P_{a,\mathbf b}(\theta\cdot\xi).\vspace{-2mm}
\end{equation*}
The prescribed zeros follow from the factors in $P_{a,\mathbf b}$.
Since $\theta\cdot\theta=0$, we have
$
\Delta\bigl(P_{a,\mathbf b}(\theta\cdot\xi)\bigr)
=P_{a,\mathbf b}''(\theta\cdot\xi)\,\theta\cdot\theta=0$. 
The gauge identity \eqref{eq:gauge-computation}, with $p_\sharp=-\mu-\frac{|A|^2}{4}$, therefore gives
\eqref{eq:isotropic-annihilator-equation}.

\smallskip\noindent
\emph{Step 3: uniform bounds and normalization.}
Since $0<c_N<1$ and $J\leq2N-1$, the choice of $\theta$ gives\vspace{-2mm}
\begin{equation*}
|H_{a,\mathbf b}(a)|
=\prod_{j=1}^J|\theta\cdot(a-b_j)|
\geq c_N^{2N-1}\prod_{j=1}^J|a-b_j|.\vspace{-2mm}
\end{equation*}
Fix a ball $B$ containing $\overline G$, chosen in terms of $G$ alone.
Since $J\leq2N-1$, $|\theta|=\sqrt2$, and $a,b_j\in G$, differentiating
the explicit formula gives\vspace{-3mm}
\begin{equation*}
\|H_{a,\mathbf b}\|_{W^{2,\infty}(B)}\leq C,\vspace{-2mm}
\end{equation*}
where $C$ depends only on $G,A,N$. Restriction to $G$ yields the
$H^2(G)$ bound in \eqref{eq:isotropic-annihilator-bounds}.

If $|a-b_j|\geq\rho>0$ for every $j$, the product lower bound gives
$|H_{a,\mathbf b}(a)|\geq c_{\rho,N}>0$ uniformly for $0\leq J\leq2N-1$.
Consequently,\vspace{-2mm}
\begin{equation*}
\|V_{a,\mathbf b}\|_{W^{2,\infty}(B)}\leq C_\rho.\vspace{-1mm}
\end{equation*}
Restriction to $G$ gives the $H^2(G)$ bound, while the mean-value formula
on the convex ball $B$ gives the Lipschitz bound in
\eqref{eq:normalized-isotropic-bounds}.
\end{proof}

\vspace{-3mm}

\subsection{Spectral synthesis for one-dimensional diffusion sources}

In one dimension, since no nonzero complex‑isotropic direction exists, the construction from \cref{lem:isotropic-annihilator} cannot be used. For diffusion sources, we therefore derive the probe estimates needed in \cref{lem:localization-matching} by combining finitely many spectral parameters.

Assume in this subsection that $G=(0,\ell)$ and $N=N_1$. Let $z$ be the
centered solution in \eqref{eq:center-equation}, and set
$\varepsilon=\|z\|_{\operatorname{obs}}$.
Set\vspace{-2mm}
\begin{equation}
\label{eq:one-dimensional-spectral-family}
p_j=j^2+Aj-\mu,\qquad \phi_j(\xi)=\mathrm e^{j\xi},
\qquad j=0,\ldots,2N-1.\vspace{-1mm}
\end{equation}
Then each $p_j$ is real and $(\mathcal L^*+p_j)\phi_j=0$.
Writing $q(\xi)=\mathrm e^\xi$, we have $\phi_j=q^j$.

By \cref{lem:sto-source-estimate},\vspace{-1mm}
\begin{equation}
\label{eq:one-dimensional-moment-control}
\|\mathfrak D[\phi_j]\|_{L^2(0,T)}\leq C\varepsilon,
\qquad j=0,\ldots,2N-1,\vspace{-1mm}
\end{equation}
where $C$ may additionally depend on $N$.

\begin{lemma}[One-dimensional annihilators by spectral synthesis]
\label{lem:one-dimensional-synthesis}
Let $a\in G$ and $b_1,\ldots,b_J\in G\setminus\{a\}$, where
$J\leq2N-1$, and define\vspace{-2mm}
\begin{equation}
\label{eq:one-dimensional-annihilator}
H_{a,\mathbf b}(\xi)
=\prod_{j=1}^J\bigl(q(\xi)-q(b_j)\bigr).\vspace{-2mm}
\end{equation}
Then $H_{a,\mathbf b}$ is a linear combination of
$\phi_0,\ldots,\phi_J$ with uniformly bounded coefficients, and\vspace{-2mm}
\begin{align}
\label{eq:one-dimensional-annihilator-control}
&\|\mathfrak D[H_{a,\mathbf b}]\|_{L^2(0,T)}\leq C\varepsilon,\\
\label{eq:one-dimensional-annihilator-lower}
&|H_{a,\mathbf b}(a)| \geq c\prod_{j=1}^J|a-b_j|.
\end{align}
If $|a-b_j|\geq\rho>0$ for all $j$, the normalized function
$V_{a,\mathbf b}=H_{a,\mathbf b}/H_{a,\mathbf b}(a)$ is a linear
combination of $\phi_0,\ldots,\phi_J$ with uniformly bounded coefficients
and satisfies\vspace{-2mm}
\begin{equation}
\label{eq:one-dimensional-normalized-bounds}
\|\mathfrak D[V_{a,\mathbf b}]\|_{L^2(0,T)}\leq C_\rho\varepsilon,
\qquad
\|V_{a,\mathbf b}\|_{\operatorname{Lip}(G)}\leq C_\rho.\vspace{-2mm}
\end{equation}
The constants may additionally depend on $N$;
the constants in the normalized estimates also depend on $\rho$.
\end{lemma}

\begin{proof}
We combine the $L^2(0,T)$ estimates \eqref{eq:one-dimensional-moment-control}
after testing at each individual $p_j$; the synthesized function need not
solve one common adjoint equation. Expanding the polynomial in
\eqref{eq:one-dimensional-annihilator} gives\vspace{-2mm}
\begin{equation*}
H_{a,\mathbf b}(\xi)=\sum_{j=0}^J c_j\phi_j(\xi).\vspace{-2mm}
\end{equation*}
The roots $q(b_j)$ remain in a fixed compact interval, so the elementary
symmetric-polynomial formula gives $\sum_{j=0}^J|c_j|\leq C$.  Linearity and
\eqref{eq:one-dimensional-moment-control} prove
\eqref{eq:one-dimensional-annihilator-control}.

The map $q$ is bi-Lipschitz on $[0,\ell]$: there exist $0<c_q\leq C_q$ such
that\vspace{-2mm}
\begin{equation*}
c_q|x-y|\leq|q(x)-q(y)|\leq C_q|x-y|,
\qquad x,y\in[0,\ell].\vspace{-2mm}
\end{equation*}
Consequently,\vspace{-3mm}
\begin{equation*}
|H_{a,\mathbf b}(a)|
=\prod_{j=1}^J|q(a)-q(b_j)|
\geq c\prod_{j=1}^J|a-b_j|,\vspace{-2mm}
\end{equation*}
which proves \eqref{eq:one-dimensional-annihilator-lower}.  If all
$|a-b_j|\geq\rho$, the denominator $H_{a,\mathbf b}(a)$ is bounded away
from zero.  The normalized polynomial coefficients and the first derivative
of the explicit formula are then uniformly bounded, proving
\eqref{eq:one-dimensional-normalized-bounds}.\vspace{-2mm}
\end{proof}

\begin{remark}
This synthesis uses the fact that $\mathfrak D[\phi_j]$ is estimated in
the same space $L^2(0,T)$ for every $p_j$, with coefficients
$\gamma_l^j$ independent of $p_j$. For drift sources, changing the
spectral parameter changes the scalar moments
$\widetilde{\lambda_k^j}(p_j)$, so these estimates do not combine into
the same annihilation identity.\vspace{-2mm}
\end{remark}

\subsection{Proof of the finite-source theorem}

\begin{proof}[Proof of \cref{thm:main-multi}]
Let $\bar v$ and $z$ solve \eqref{eq:mean-equation} and
\eqref{eq:center-equation}, respectively.

\smallskip\noindent
\textbf{Drift sources ($n=2,3$).}
Set $N=N_0$, $\varepsilon=\|\bar v\|_{\operatorname{obs}}$, $E=\mathbb C$,
$a_k^j=x_k^j$, and 
$c_k^j=\widetilde{\lambda_k^j}(p_\sharp)
=\int_0^T\mathrm e^{-p_\sharp t}\lambda_k^j(t)\,dt$. 
Using the node lists in \cref{sec:matching-principle},
\cref{lem:det-source-estimate,lem:isotropic-annihilator} at $p=p_\sharp$
give the product-probe bound in \eqref{eq:abstract-product-probe}.
Assumption~\ref{ass:multi-drift} supplies separation and the coefficient
lower bound. Thus \cref{lem:localization-matching} gives the drift location
estimate, without an upper bound on the strengths.
Under \eqref{eq:multi-drift-upper},\vspace{-3mm}
\begin{equation*}
\sum_{k=1}^{N_0}\bigl(|c_k^1|+|c_k^2|\bigr)
\leq\|\mathrm e^{-p_\sharp\cdot}\|_{L^2(0,T)}M.\vspace{-2mm}
\end{equation*}
The same probe estimates with $\rho=3\delta/4$ give
\eqref{eq:abstract-normalized-probe}, so the coefficient conclusion of
\cref{lem:localization-matching} yields the weighted-moment estimate
for the same permutation.

\smallskip\noindent
\textbf{Diffusion sources.}
Set $N=N_1$, $\varepsilon=\|z\|_{\operatorname{obs}}$,
$E=L^2(0,T;\mathbb C)$, $a_l^j=y_l^j$, and $c_l^j=\gamma_l^j$,
viewing the real strengths in their complexification.
For $n=2,3$, \cref{lem:isotropic-annihilator,lem:sto-source-estimate}
at $p=p_\sharp$ give the product and normalized probe bounds.
For $n=1$, \cref{lem:one-dimensional-synthesis} gives the same bounds.
In both cases, use the prescribed node lists and $\rho=3\delta/4$ for
the normalized probes. Assumption~\ref{ass:multi} supplies separation
and the coefficient bounds, so \cref{lem:localization-matching} gives
the diffusion location and strength estimates with the same permutation.
Finally, \eqref{eq:channel-contraction} bounds both channel errors by
$\|u^1-u^2\|_{\operatorname{obs}}$ and proves all three conclusions.
\end{proof}

\begin{remark}
\label{rem:deterministic-matching}
In the planar deterministic setting of \cite{HuangJinKianTriki2026},
with $\mu=-|A|^2/4$, the moments at $p_\sharp=0$ are unweighted.
The nearest-point separation bound \eqref{eq:matching-nearest-separation}
in the proof of \cref{lem:localization-matching} strengthens the intermediate
product estimate
in their Section~4.2, (4.11), directly to Lipschitz localization and
constructs the matching permutation. This improvement uses the separation
assumption already present in that argument.
\end{remark}
\vspace{-3mm}

\section{Numerical experiments}
\label{sec:numerics}

We present three experiments illustrating the stability recovery mechanisms: simultaneous reconstruction of one source per channel, temporal conditioning at known locations, and recovery of two diffusion sources as their separation decreases.

The procedure combines channel separation, dictionary-based localization, and regularized least-squares recovery of temporal strengths. It is independent of the adjoint probe constructions used in the proofs. The experiments focus on recoverability and conditioning, not on estimating stability exponents.

\vspace{-2mm}

\subsection{Discretization and reconstruction}

Experiments use the unit disk \(G\subset\mathbb R^2\), \(u_0=0\), with \(T=1\), \(T_0=0.55\), \(T_1=0.75\), \(A=(0.35,-0.20)\), \(\mu=0.5\), and observation arc \(\Gamma=\{(\cos\theta,\sin\theta):-\pi/3<\theta<\pi/2\}\).

Let $\mathcal T_h$ be a triangular mesh with spatial mesh size
$h=\max_{E\in\mathcal T_h}\operatorname{diam}(E)$. We use continuous
$P_1$ finite elements on $\mathcal T_h$ with nodal basis $\{\varphi_i\}$
and implicit Euler time stepping. Let $M_h$ and $K_h$ denote the mass
matrix and the finite element matrix associated with
$\mathcal L=-\Delta+A\cdot\nabla+\mu$, respectively, and set
$b_h(\xi)_i=\varphi_i(\xi)$. With $U^n$ at $t_n=n\Delta t$,
$\Delta t=T/N_t$, and $\Delta W_n=W(t_{n+1})-W(t_n)$, the scheme for one
source per channel is\vspace{-2mm}
\[
(M_h+\Delta t K_h)U^{n+1}=M_hU^n+\Delta t\,\lambda(t_{n+1})b_h(x)+\gamma(t_n)b_h(y)\Delta W_n.\vspace{-2mm}
\]
Location data use a \(29{,}958\)-triangle mesh with 480 time steps; inversion uses \(8{,}068\) triangles and 160 observation steps. Sources are off nodes.

The discrete norm approximates \(\|\cdot\|_{\operatorname{obs}}\). For boundary residuals $r_n$, let $M_\partial,M_\Gamma$ be the mass matrices on $\partial G,\Gamma$, respectively, and $S_\Gamma$ the tangential stiffness matrix on $\Gamma$; with $D_t r_n=(r_n-r_{n-1})/\Delta t$ for $n\geq1$, we set\vspace{-3mm}
\begin{equation}\label{eq:num-observation-norm}
\|r\|_h^2
:= \Delta t\sum_{n=0}^{N_t}r_n^\top M_{\partial}r_n
+\Delta t\sum_{t_n\in[T_1,T]}
\bigl[r_n^\top M_\Gamma r_n
+(D_t r_n)^\top M_\Gamma D_t r_n
+r_n^\top S_\Gamma r_n\bigr].
\vspace{-2mm}
\end{equation}
For $N_{\mathrm{MC}}$ paths, the empirical norm is\vspace{-4mm}
\begin{equation}
\label{eq:num-empirical-norm}
\|r\|_{h,N_{\mathrm{MC}}}^2
=\frac1{N_{\mathrm{MC}}}\sum_{m=1}^{N_{\mathrm{MC}}}
\|r^{(m)}\|_h^2.\vspace{-2mm}
\end{equation}
The same weights are used for calibration, reconstruction, and conditioning. Noise is smoothed and scaled to relative level \(\eta\).

Each realization is paired with its Brownian path. Expectation/centering separates channels: reconstruct diffusion from centered observations, subtract its mean response, then recover drift from the corrected mean.

For fixed locations, response is linear in strengths. Each strength uses 18 cubic B-splines on \([0,T_0]\), extended by zero, with Tikhonov regularization. For candidates \(\theta\), coefficients \(a\), response matrix \(B(\theta)\), data \(\mathbf d\), and weight matrix \(W_{\rm obs}\), set \(\|r\|_{W_{\rm obs}}^2=r^*W_{\rm obs}r\). The reconstruction is\vspace{-2mm}
\begin{equation}
\label{eq:num-profile-fit}
a_\alpha(\theta)=\operatorname*{argmin}_a
\bigl\{\|B(\theta)a-\mathbf d\|_{W_{\rm obs}}^2
+\alpha\|Pa\|_2^2+\beta_{\rm stab}(\theta)\|a\|_2^2\bigr\},\vspace{-2mm}
\end{equation}
with \(P\) the second-difference matrix and \(\beta_{\rm stab}\) a diagonal stabilizer. Diffusion data and responses are centered. Regularization is chosen by residual matching.

Localization projects observations onto precomputed candidate responses: diffusion uses centered data correlated with Brownian increments and low-boundary Fourier modes; drift uses the corrected mean projected onto spline responses. Candidates with smallest residuals are refined locally; final drift location minimizes the full residual.

Single-source diffusion uses all 480 increments; localization and two-source recovery use pathwise-aggregated increments on coarser grids.

Each configuration uses 30 independent runs, separate from tuning. Location errors are normalized by $\operatorname{diam}(G)=2$, and profile errors are reported as relative $L^2(0,T_0)$ errors. Median errors are shown, with variability indicated in the figures and tables.

\vspace{-2mm}

\subsection{Simultaneous single-source recovery}

In the first experiment, \(N_0=N_1=1\) (indices suppressed), with drift and diffusion sources at \(x^\star=(-0.38,0.18)\) and \(y^\star=(0.32,-0.24)\). Set \(\rho(t)=\sin^2(\pi t/T_0)\) on \([0,T_0]\) and zero on \((T_0,T]\). Before \(L^2(0,T_0)\) normalization,\vspace{-2mm}
\[
\lambda_{\rm raw}(t)=\rho(t)\bigl[1+0.25\cos(4\pi t/T_0)-0.15\sin(6\pi t/T_0)\bigr],\quad
\gamma_{\rm raw}(t)=\rho(t)\bigl[0.80\sin(4\pi t/T_0)-0.45\sin(10\pi t/T_0)\bigr].\vspace{-2mm}
\]
Both normalized profiles have unit \(L^2(0,T_0)\) norm and are outside the spline space. With \(N_{\mathrm{MC}}=512\) paths and \(\eta\in\{0.25\%,0.5\%,1\%,2\%\}\), \Cref{fig:num-single} shows recovered locations and profiles and their error dependence on noise; the representative run is the one closest to the sample median at \(1\%\) noise.

\begin{figure}[!htb]
\centering
\includegraphics[width=.8\textwidth]{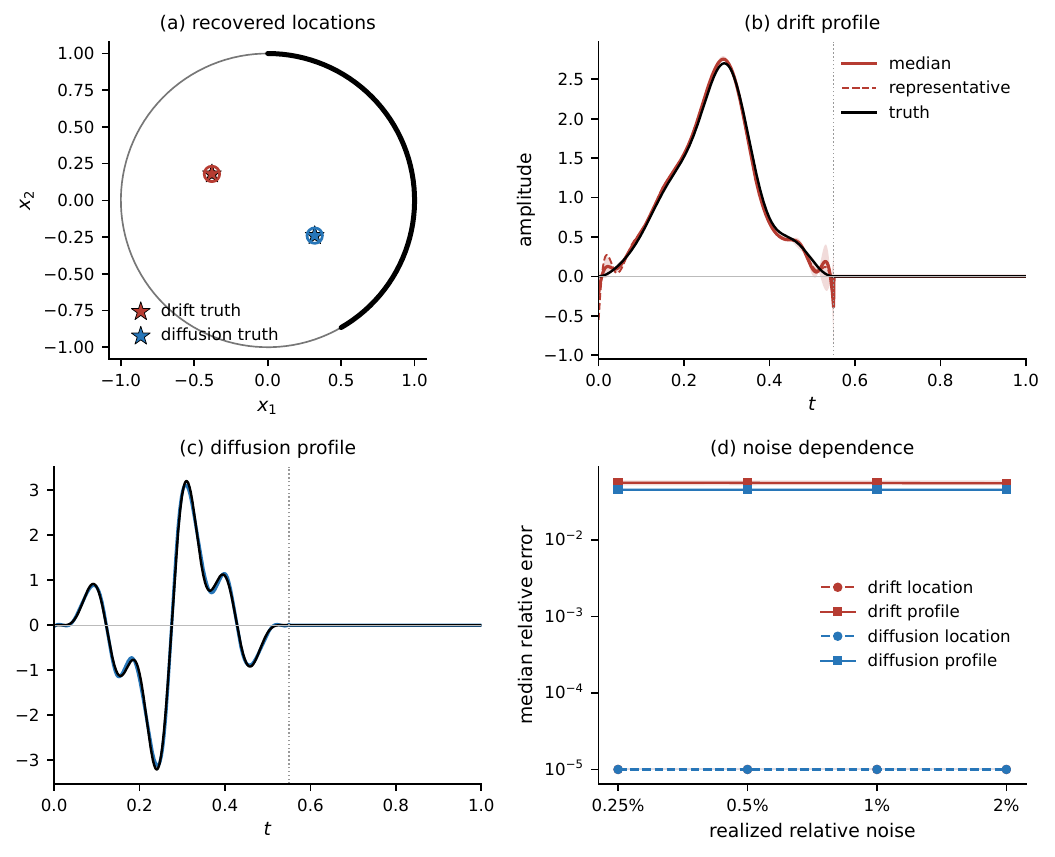}
\caption{Simultaneous single-source recovery. Panel (a) shows the true
and reconstructed locations at $1\%$ noise. Panels (b)--(c) compare
the true profiles with the pointwise medians, empirical
$10\%$--$90\%$ bands, and a representative run (dashed).
Panel (d) shows median errors against noise level, with $95\%$
bootstrap intervals and a plotting floor of $10^{-5}$.}
\label{fig:num-single}\vspace{-3mm}
\end{figure}

At \(1\%\) noise, both normalized location errors are below \(0.05\) across all \(30\) runs. Median relative profile errors are \(0.055\) for the drift and \(0.045\) for the diffusion (see \cref{tab:num-single}). The vanishing median location errors indicate that the true points are included in the refined search dictionary. The reconstructed profiles capture the drift variation and the sign changes of the diffusion strength, confirming that the two channels can be separated and recovered from the same noisy observations.

\begin{table}[htbp]
\centering
\small
\begin{tabular}{lrrrr}
\hline
Channel & \multicolumn{2}{c}{Location error} & \multicolumn{2}{c}{Profile error} \\
& Median & \shortstack{90th\\percentile} & Median & \shortstack{90th\\percentile} \\
\hline
drift & $<10^{-12}$ & $2.5\times10^{-3}$ & $0.055$ & $0.076$ \\
diffusion & $<10^{-12}$ & $<10^{-12}$ & $0.045$ & $0.046$ \\
\hline
\end{tabular}
\caption{Single-source errors over $30$ runs at
$N_{\mathrm{MC}}=512$ and $1\%$ relative noise. Location errors are
normalized by the domain diameter, and profile errors are relative
$L^2(0,T_0)$ errors; values below $10^{-12}$ are reported at that threshold.}
\label{tab:num-single}
\end{table}

The reconstruction is qualitatively unchanged on the three inversion
meshes in \cref{tab:num-mesh}. These have $3{,}017$, $8{,}068$, and
$10{,}023$ triangles, with $120$, $160$, and $240$ observation time
steps, respectively. Over ten runs at $0.25\%$ noise, the median location
errors remain below $10^{-12}$. The drift-profile error decreases under
refinement, while the diffusion-profile error remains near $0.045$.

\begin{table}[htbp]
\centering
\small
\begin{tabular}{lrrr}
\hline
Mesh & Drift profile & Diffusion profile & Forward discrepancy \\
\hline
coarse & $0.09198$ & $0.04530$ & $1.4063\times10^{-3}$ \\
standard & $0.05706$ & $0.04531$ & $6.6425\times10^{-4}$ \\
fine & $0.05302$ & $0.04538$ & $7.2141\times10^{-4}$ \\
\hline
\end{tabular}
\caption{Median relative profile errors over ten runs at $0.25\%$
noise. The forward discrepancy compares data and inversion
discretizations at the true sources using the noiseless $512$-path
ensemble of the first seed.}
\label{tab:num-mesh}\vspace{-3mm}
\end{table}

\vspace{-2mm}

\subsection{Temporal information in the two channels}

Now fix locations at the true points and compare temporal oscillation transmission in the two forward maps, both using the standard mesh and \(160\) time steps. For \(k=1,\dots,24\), define \(q_k(t)=c_k\rho(t)\sin(k\pi t/T_0)\) with \(\|q_k\|_{L^2(0,T_0)}=1\). Generalized singular values measure conditioning relative to the input norm. With \(B_0\) the drift response, \(B_1^{(m)}\) the diffusion response for path \(m\), and \((G_t)_{ij}=(q_i,q_j)_{L^2(0,T_0)}\) (trapezoidal quadrature), the response Gram matrices are\vspace{-4mm}
\[
H_0=B_0^*W_{\rm obs}B_0,\qquad
H_1=\frac1{2048}\sum_{m=1}^{2048}(B_1^{(m)})^*W_{\rm obs}B_1^{(m)}.\vspace{-2mm}
\]
Generalized singular values \(\sigma\) satisfy \(H_jv=\sigma^2G_tv\); their largest-to-smallest ratio gives the condition number on the span of the first \(K\) modes. Individual mode gains \(\sqrt{(H_j)_{kk}}\) are normalized by the gain of \(q_1\).

To connect to reconstruction, use the same oscillatory profile in both channels:\vspace{-2mm}
\[
q_H(t)=c_H\rho(t)[\sin(4\pi t/T_0)+0.6\sin(20\pi t/T_0)],\vspace{-2mm}
\]
normalized to unit \(L^2(0,T_0)\) norm. Reconstruction uses the span of \(q_1,\dots,q_{24}\) (containing \(q_H\)) with an identity penalty in \eqref{eq:num-profile-fit}. Drift data are a single deterministic trace; diffusion data use \(128\) fixed paths. For each channel, we take \(100\) independent noise realizations at \(0.1\%,0.3\%,1\%,3\%\), calibrated in that channel's norm.

\begin{figure}[htbp]
\centering
\includegraphics[width=.90\textwidth]{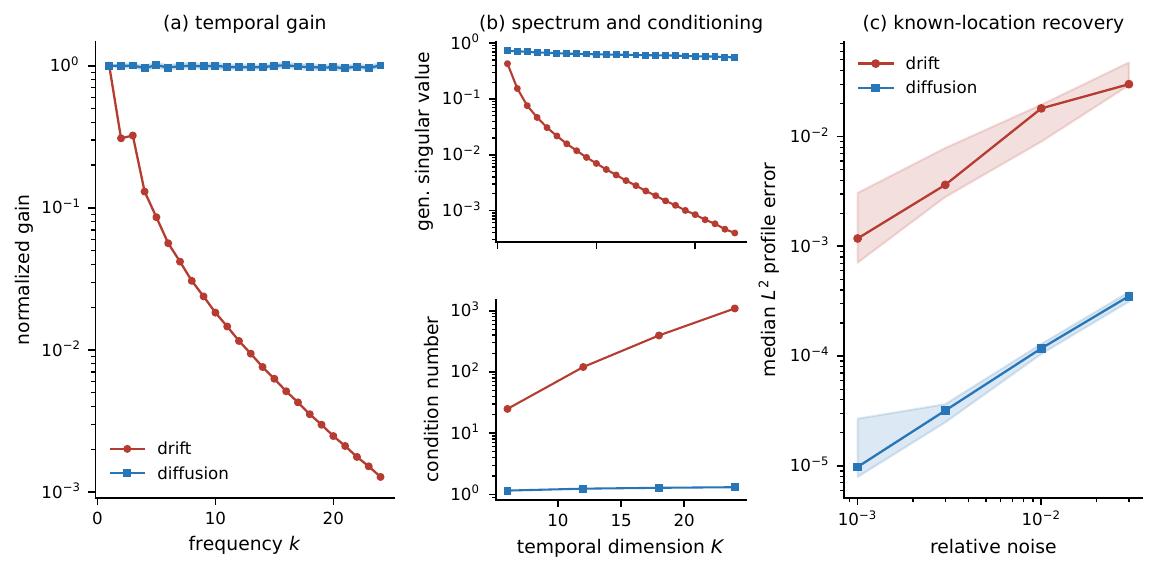}
\caption{Temporal information in the two channels: (a) gains normalized by first-mode gain, with \(95\%\) bootstrap intervals for diffusion; (b) \(K=24\) singular values (nonincreasing) and condition numbers; (c) median relative \(L^2\) errors and interquartile bands over \(100\) noise realizations for the common high-frequency profile at known locations.}
\label{fig:num-mechanism}\vspace{-3mm}
\end{figure}

\Cref{fig:num-mechanism} shows a pronounced loss of sensitivity to
temporal oscillations in the drift channel. At $k=20$, its normalized
gain is $0.00248$, compared with $0.975$ for the diffusion channel.
At $K=24$, the drift condition number is $833$ times the diffusion
condition number. This difference is also visible in reconstruction:
at $1\%$ noise, the median profile errors are $0.0181$ and
$1.18\times10^{-4}$, respectively. Temporal oscillations can cancel in
the deterministic time integral, while It\^o's isometry retains the
energy of the diffusion integrand. The three comparisons illustrate this
distinction underlying the strength estimates in \cref{thm:main-single}.

\vspace{-2mm}

\subsection{Two-source recovery and separation}

The final experiment examines separation effects on localization and profile reconstruction. We take \(N_0=0\), \(N_1=2\), \(y_{\rm c}=(0.08,-0.08)\), \(\varrho=(\cos35^\circ,\sin35^\circ)\), \(y_1^\star=y_{\rm c}-d\varrho/2\), \(y_2^\star=y_{\rm c}+d\varrho/2\), and \(d\in\{0.50,0.35,0.25,0.18,0.12\}\).

Define the unnormalized profiles by\vspace{-2mm}
\[
\gamma_{1,\rm raw}(t)=\rho(t)[0.85\sin(4\pi t/T_0)+0.30\cos(8\pi t/T_0)],\quad\;
\gamma_{2,\rm raw}(t)=\rho(t)[0.75\cos(6\pi t/T_0)-0.40\sin(12\pi t/T_0)].\vspace{-2mm}
\]
We normalize these profiles to have $L^2(0,T_0)$ norms $1$ and $0.8$, respectively. Each separation uses \(N_{\mathrm{MC}}=512\), \(1\%\) noise, and \(30\) runs.

Joint localization fixes the count at two, searches dictionary pairs with separation \(\ge0.08\), refines locally, and reconstructs profiles. Dimensionless separation \(d/h_{\rm inv}\), with \(h_{\rm inv}=\sqrt{2\pi/8000}\simeq0.028025\), corresponds to the inversion mesh target. 
Evaluation matches locations via the Hungarian algorithm, using the same permutation \(\sigma\) for signed profile errors; success is \(\max_{l=1,2}|\widehat y_{\sigma(l)}-y_l^\star|<d/4\).

The condition number uses generalized singular values of the two-source response and block temporal Gram matrices at true locations. \Cref{fig:num-multi} shows matched reconstructions and their separation dependence.

\begin{figure}[!htb]
\centering
\includegraphics[width=.76\textwidth]{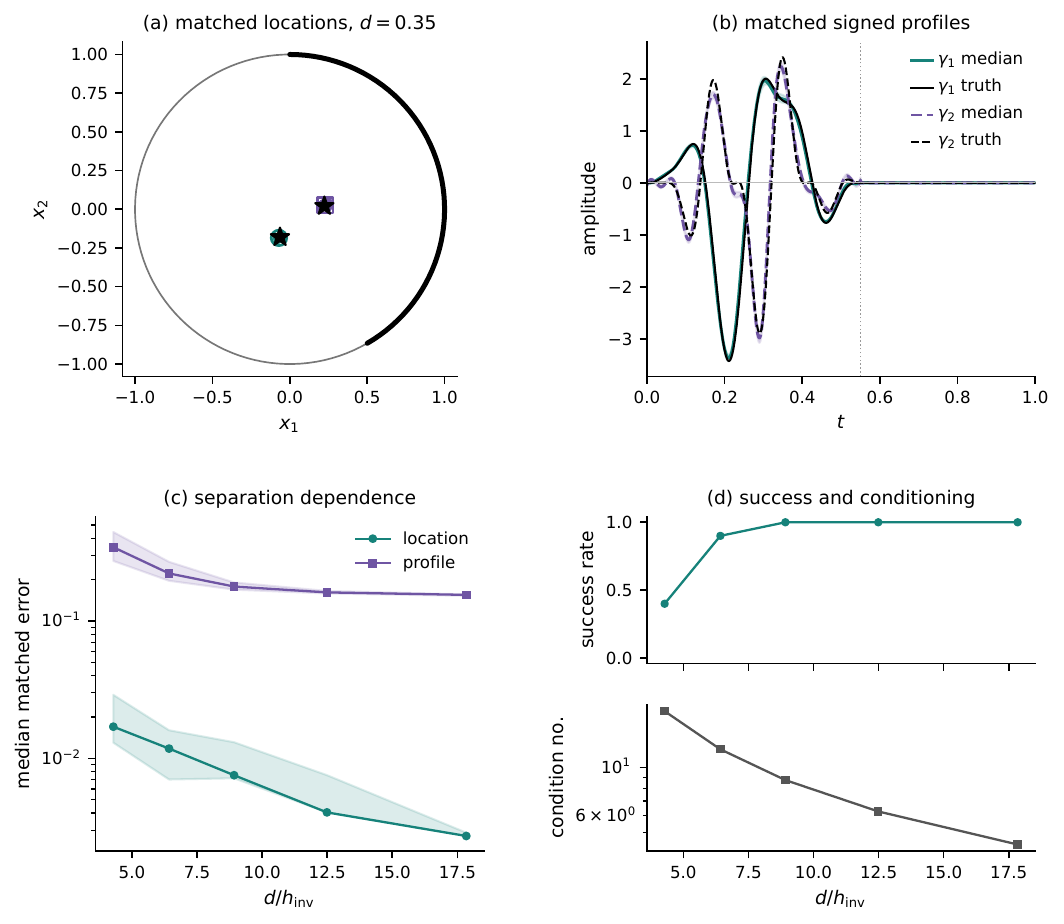}
\caption{Two-source diffusion recovery: (a)--(b) matched locations and signed profiles at \(d=0.35\) with \(10\%\)–\(90\%\) bands; (c) median diameter-normalized location and relative profile errors vs. \(d/h_{\rm inv}\), with interquartile bands; (d) matching success rate (threshold \(d/4\)) and condition number.
}
\label{fig:num-multi}\vspace{-2mm}
\end{figure}

\Cref{tab:num-multi} shows successful matching in all $90$ runs with
$d\geq0.25$. As the separation decreases, the median location and profile
errors increase, together with the condition number. The success rate
decreases to $90\%$ at $d=0.18$ and to $40\%$ at $d=0.12$,
where the median relative profile error is $0.344$. Nearby sources
produce less distinguishable boundary responses; the resulting deterioration
illustrates the role of the separation assumption in \cref{thm:main-multi}.

\begin{table}[!htb]
\centering
\small
\renewcommand{\arraystretch}{1.1}
\begin{tabular}{rrrrrrr}
\toprule
\multicolumn{1}{c}{\multirow{2}{*}{$d/h_{\rm inv}$}}
& \multicolumn{1}{c}{\multirow{2}{*}{Success}}
& \multicolumn{2}{c}{Location error}
& \multicolumn{2}{c}{Profile error}
& \multicolumn{1}{c}{\multirow{2}{*}{\shortstack{Condition\\number}}} \\
\cmidrule(lr){3-4}\cmidrule(lr){5-6}
& & \multicolumn{1}{c}{Median} & \multicolumn{1}{c}{90th percentile}
& \multicolumn{1}{c}{Median} & \multicolumn{1}{c}{90th percentile} & \\
\midrule
$17.8$ & $100\%$ & $0.003$ & $0.006$ & $0.155$ & $0.157$ & $4.4$ \\
$12.5$ & $100\%$ & $0.004$ & $0.010$ & $0.161$ & $0.169$ & $6.3$ \\
$8.9$ & $100\%$ & $0.008$ & $0.017$ & $0.178$ & $0.221$ & $8.7$ \\
$6.4$ & $90\%$ & $0.012$ & $0.022$ & $0.222$ & $0.308$ & $12.1$ \\
$4.3$ & $40\%$ & $0.017$ & $0.040$ & $0.344$ & $0.475$ & $18.2$ \\
\bottomrule
\end{tabular}
\caption{Two-source results: 30 runs per separation, \(N_{\mathrm{MC}}=512\), \(1\%\) noise. Errors are maxima over matched sources; locations normalized by domain diameter, profiles relative \(L^2(0,T_0)\). Success threshold \(d/4\); condition number at true locations reported as median.
}
\label{tab:num-multi}
\end{table}

\FloatBarrier
\appendix

\vspace{-3mm}

\section{The direct problem and boundary regularity}
\label{app:direct}

\vspace{-2mm}

We prove \cref{prop:direct-wellposed} using standard semigroup estimates,
with particular attention to the full-time boundary trace: the spatial
separation of the sources and the boundary makes the localized stochastic
convolution square integrable. Smoothing after the sources become
inactive gives the late-time regularity. We then justify the point-source
approximation used in \cref{lem:sto-source-estimate}. Constants may depend
on $G,A,\mu,T,N_0,N_1$; other dependencies are stated explicitly.

\begin{proof}[Proof of \cref{prop:direct-wellposed}]
The adjoint of the Neumann realization \eqref{eq:A-def} is\vspace{-2mm}
\begin{equation*}
\mathcal A^*v=\Delta v+A\cdot\nabla v-\mu v,
\qquad D(\mathcal A^*)=
\{v\in H^2(G):\partial_\nu v+(A\cdot\nu)v=0\}.\vspace{-2mm}
\end{equation*}
Elliptic regularity \cite[Theorem~2.3.3.2]{Grisvard2011}, after absorbing lower-order terms, gives equivalence between the graph norm of \(D(\mathcal A^*)\) and the \(H^2(G)\) norm. Since \(n\leq3\), Sobolev embedding \(H^2(G)\hookrightarrow C(\overline G)\) yields\vspace{-2mm}
\begin{equation}\label{eq:delta-uniform}
\sup_{y\in G}\|\delta_y\|_{\mathcal X_{-2}}\leq C.
\vspace{-2mm}
\end{equation}

Standard semigroup and extrapolation theory \cite[Chapter~3]{Pazy1983}, \cite[Section~II.5]{EngelNagel2000} provide the analytic semigroups \(S(t)\) and \(S_{-2}(t)\). Multiplication by \(\mathrm e^{A\cdot x/2}\) makes \(\mathcal A\) similar to a self-adjoint Robin realization, so \(S_{-2}(t)\) satisfies \(\|S_{-2}(t)\|\leq\mathrm e^{\rho t}\) for some \(\rho\geq0\) in an equivalent Hilbert norm. Applying \cite[Theorems~3.18 and~3.20]{Lue2021a} on \(\mathcal X_{-2}\) with \(p=2\) and \eqref{eq:delta-uniform} yields the unique mild solution and \eqref{eq:wellposed-est}, where\vspace{-3mm}
\begin{equation}\label{eq:mild}
u(t)=S(t)u_0
+\sum_{k=1}^{N_0}\int_0^t S_{-2}(t-s)\lambda_k(s)\delta_{x_k}\,ds
+\sum_{l=1}^{N_1}\int_0^t S_{-2}(t-s)\gamma_l(s)\delta_{y_l}\,dW(s)
\quad\text{in }\mathcal X_{-2}.
\vspace{-2mm}
\end{equation}
Analytic smoothing on the extrapolation scale \cite[Theorem~2.6.13]{Pazy1983}, after a spectral shift, gives the estimates used below:\vspace{-2mm}
\begin{align}
&\|S(t)\|_{\mathcal L(L^2(G),H^2(G))} \leq Ct^{-1}, \qquad 0<t\leq2T, \notag\\
\label{eq:positive-lag-smoothing}
&\sup_{\tau\leq t\leq T}\sum_{j=0}^1
\|\partial_t^jS_{-2}(t)\|_{\mathcal L(\mathcal X_{-2},H^2(G))}
\leq C_\tau,\qquad 0<\tau\leq T.
\end{align}
Write \(K_y(t)=k(t,\cdot,y)=S_{-2}(t)\delta_y\) for the continuous interior representative of the heat kernel. The Gaussian bound in \cite[Theorems~4.4 and~6.1]{Daners2000HeatKernel} and spatial integration give, for \(0<t\leq2T\),\vspace{-2mm}
\begin{align}
&|k(t,\xi,y)| \leq Ct^{-n/2}\exp\biggl(-\frac{|\xi-y|^2}{ct}\biggr),
\label{eq:gaussian-kernel}\\
&\sup_{y\in G}\|K_y(t)\|_{L^2(G)} \leq Ct^{-n/4}.
\label{eq:global-delta-smoothing}\vspace{-2mm}
\end{align}

Let \(u_{\rm d}\) denote the deterministic source term in \eqref{eq:mild}. Since \(n\leq3\), \eqref{eq:global-delta-smoothing} and Young's inequality give \(u_{\rm d}\in L^2(0,T;L^2(G))\). By semigroup duality, it satisfies the Neumann transposition identity used in \cite[Section~3]{HuangJinKianTriki2026}; uniqueness then identifies the two solutions. The same identification applies to \eqref{eq:mean-equation}.

Assume Condition~\ref{con1} holds. The deterministic off-source regularity follows from the proof of \cite[Theorem~3.1]{HuangJinKianTriki2026}. It remains to establish the full-time boundary bound for the stochastic term. Enclose the sources in a compact set \(\mathcal C_{\rm src}\Subset G\). Choose \(\chi\in C^\infty(\overline G)\) with \(0\leq\chi\leq1\), equal to one near \(\partial G\) and zero near \(\mathcal C_{\rm src}\), and set \(\mathbf d_\chi=\operatorname{dist}(\operatorname{supp}\chi,\mathcal C_{\rm src})>0\). From \eqref{eq:gaussian-kernel} and \eqref{eq:global-delta-smoothing}, together with the semigroup property and analytic smoothing, we have for \(y\in\mathcal C_{\rm src}\) and \(0<t\leq T\),\vspace{-2mm}
$$
\|\chi K_y(t)\|_{L^2(G)}  \leq Ct^{-n/4}\mathrm e^{-c\mathbf d_\chi^2/t},\qquad
\|\chi K_y(t)\|_{H^2(G)}  \leq C_\chi t^{-1-n/4}.\vspace{-2mm}
$$
Standard interpolation then yields\vspace{-2mm}
\begin{equation}\label{eq:point-kernel-offsource}
\sup_{y\in\mathcal C_{\rm src}}\|\chi K_y(t)\|_{H^1(G)}
\leq C_\chi t^{-(n+2)/4}\mathrm e^{-c\mathbf d_\chi^2/t},
\qquad 0<t\leq T.
\vspace{-2mm}
\end{equation}
The right-hand side of \eqref{eq:point-kernel-offsource} lies in \(L^2(0,T)\). For each diffusion source, Itô's isometry and Tonelli's theorem give\vspace{-2mm}
\[
\mathbb E\int_0^T
\Bigl\|\int_0^t\chi K_y(t-s)\gamma(s)\,dW(s)\Bigr\|_{H^1(G)}^2\,dt
\leq\|\gamma\|_{L^2(0,T)}^2
\int_0^T\|\chi K_y(r)\|_{H^1(G)}^2\,dr.\vspace{-2mm}
\]
Combining this with Young's inequality for \(u_{\rm d}\), and writing \(v=u-S(\cdot)u_0\), yields\vspace{-2mm}
\begin{equation}\label{eq:offsource-est}
\mathbb E\|\chi v\|_{L^2(0,T;H^1(G))}^2
\leq C_\chi\Bigl(\sum_{k=1}^{N_0}\|\lambda_k\|_{L^2(0,T)}^2
+\sum_{l=1}^{N_1}\|\gamma_l\|_{L^2(0,T)}^2\Bigr).
\vspace{-2mm}
\end{equation}
Since \(\chi=1\) near \(\partial G\), taking traces gives the full-time boundary bound, with a constant depending on the source-to-boundary distances. For the initial component, analytic smoothing, interpolation, and the multiplicative trace inequality give the integrable bound\vspace{-2mm}
\[
\|S(t)u_0\|_{L^2(\partial G)}^2
\leq C\|S(t)u_0\|_{L^2(G)}\|S(t)u_0\|_{H^1(G)}
\leq Ct^{-1/2}\|u_0\|_{L^2(G)}^2.\vspace{-2mm}
\]

Finally, since the sources vanish after \(T_0\), we have \(u(t)=S_{-2}(t-T_0)u(T_0)\) for \(t\geq T_0\). Thus \eqref{eq:positive-lag-smoothing} with \(\tau=T_1-T_0\), together with \eqref{eq:wellposed-est}, gives \(u\in L^2(\Omega;H^1(T_1,T;H^2(G)))\) and the terminal bound, uniformly in the source locations. Taking traces and combining the preceding bounds proves \eqref{eq:boundary-trace-estimate}.
\end{proof}

\begin{lemma}[Point-source approximation]
\label{lem:point-source-approximation}
Under Condition~\ref{con1}, let $z$ and $z_\epsilon$ solve
\eqref{eq:center-equation} and \eqref{eq:regularized-center-equation},
respectively. Then\vspace{-2mm}
\begin{equation}
\label{eq:point-source-convergence}
\lim_{\epsilon\to0}\big(\|z_\epsilon-z\|_{\operatorname{obs}}
+\|z_\epsilon(T)-z(T)\|_{L^2_{\mathcal F_T}(\Omega;H^2(G))}\big)
=0.\vspace{-1mm}
\end{equation}
Moreover, for every $v\in H^2(G;\mathbb C)$ and fixed $p\in\mathbb R$,\vspace{-2mm}
\begin{equation*}
\lim_{\epsilon\to0}\int_0^T\mathrm e^{-pt}\mathfrak D_\epsilon[v](t)\,dW(t)
=
\int_0^T\mathrm e^{-pt}\mathfrak D[v](t)\,dW(t)
\quad\text{in }L^2(\Omega;\mathbb C).\vspace{-2mm}
\end{equation*}
\end{lemma}

\begin{proof}
For the mollifiers in \eqref{eq:source-mollifier}, Sobolev embedding \(H^2(G)\hookrightarrow C^{0,\alpha}(\overline G)\) for some \(\alpha>0\) yields\vspace{-2mm}
\[
\|\rho_\epsilon^y-\delta_y\|_{\mathcal X_{-2}}
\leq C\epsilon^\alpha \longrightarrow 0.\vspace{-2mm}
\]
Since the sources vanish after \(T_0\), the smoothing estimate \eqref{eq:positive-lag-smoothing} with \(\tau=T_1-T_0\), together with Itô's isometry, gives
\(z_\epsilon\to z\) in \(L^2(\Omega;H^1(T_1,T;H^2(G)))\) and
\(z_\epsilon(T)\to z(T)\) in \(L^2(\Omega;H^2(G))\).

For the full-time boundary trace, fix a compact set \(\mathcal C_{\rm src}\Subset G\) containing all mollifier supports for sufficiently small \(\epsilon\), and choose a cutoff \(\chi\) as above. The kernel representation gives\vspace{-2mm}
\[
S(t)\rho_\epsilon^y
=\int_{\mathcal C_{\rm src}}K_\eta(t)\rho_\epsilon^y(\eta)\,d\eta.\vspace{-2mm}
\]
Since the mollifiers are nonnegative with unit mass, Minkowski's inequality and \eqref{eq:point-kernel-offsource} bound \(\|\chi S(t)\rho_\epsilon^y\|_{H^1(G)}\) uniformly in \(\epsilon\) by the same \(L^2(0,T)\) majorant. At each fixed \(t>0\), \eqref{eq:positive-lag-smoothing} gives \(S(t)\rho_\epsilon^y\to K_y(t)\) in \(H^2(G)\). Dominated convergence therefore yields\vspace{-2mm}
\[
\lim_{\epsilon\to0}\int_0^T\|\chi(S(t)\rho_\epsilon^y-K_y(t))\|_{H^1(G)}^2\,dt=0.\vspace{-2mm}
\]
Applying It\^o's isometry gives \(\chi(z_\epsilon-z)\to0\) in \(L^2(\Omega\times(0,T);H^1(G))\). The trace theorem, together with the late-time and terminal convergence, proves \eqref{eq:point-source-convergence}.

Finally, the same Sobolev embedding gives
\((\rho_\epsilon^{y_l^j},v)_{L^2(G)}\to v(y_l^j)\), and hence
\(\mathfrak D_\epsilon[v]\to\mathfrak D[v]\) in \(L^2(0,T)\).
It\^o's isometry, combined with the bounded weight \(\mathrm e^{-pt}\), gives the remaining convergence.
\end{proof}

\vspace{-5mm}

\bibliographystyle{siam-paper-no}
\bibliography{bibitems}

\end{document}